\documentclass[onefignum,onetabnum]{siamart190516}

\usepackage[utf8]{inputenc}
\usepackage{lipsum}
\usepackage{amsfonts}
\usepackage{amssymb}
\usepackage{graphicx}
\usepackage{psfrag}
\usepackage{epstopdf}
\usepackage{bm}
\usepackage[noend]{algpseudocode}
\usepackage{multirow}
\usepackage{array}
\usepackage{pbox}
\usepackage{array}
\usepackage{makecell}
\usepackage{siunitx}
\usepackage{mathtools}
\usepackage{dsfont}
\usepackage{stmaryrd}
\usepackage{subcaption}
\usepackage{mathrsfs}
\usepackage[scr=boondox]{mathalfa}
\usepackage{xcolor}
\usepackage{makecell}
\definecolor{purple}{rgb}{0.3,0.0,.4}

\ifpdf
  \DeclareGraphicsExtensions{.eps,.pdf,.png,.jpg}
\else
  \DeclareGraphicsExtensions{.eps}
\fi

\newif\iftutorial
\tutorialtrue

\newcommand{\R}[0]{\mathbb{R}}

\newcommand{\Hc}[0]{\mathcal{H}}
\newcommand{\normiii}[1]{{\left\vert\kern-0.25ex\left\vert\kern-0.25ex\left\vert #1
    \right\vert\kern-0.25ex\right\vert\kern-0.25ex\right\vert}}

\def\P{{\mathbb{P}}} 
\def\E{{\mathbb{E}}} 
\DeclareMathOperator*{\argmin}{argmin}
\newcommand{\tr}{\textrm{tr}}

\algrenewcommand\algorithmicrequire{\textbf{Input:}}
\algrenewcommand\algorithmicensure{\textbf{Output:}}

\renewcommand{\Comment}[1]{%
\hfill \mbox{$\triangleright$\,#1}}
\algnewcommand\algorithmicstorage{\textbf{Storage:}}
\algnewcommand\Storage{\item[\algorithmicstorage]}

\algdef{SE}[SUBALG]{Indent}{EndIndent}{}{\algorithmicend\ }%
\algtext*{Indent}
\algtext*{EndIndent}

\newsiamremark{remark}{Remark}
\newsiamremark{hypothesis}{Hypothesis}
\crefname{hypothesis}{Hypothesis}{Hypotheses}
\newsiamthm{claim}{Claim}

\headers{Nystr{\"o}m PCG}{Frangella, Tropp, and Udell}

\title{Randomized Nystr{\"o}m Preconditioning\thanks{Submitted to the editors DATE.
\funding{ZF and MU were supported by
NSF Award IIS-1943131,
the ONR Young Investigator Program,
and the Alfred P. Sloan Foundation.
JAT was supported by ONR BRC Award N00014-18-1-2363 and NSF FRG Award 1952777.}}}

\author{Zachary Frangella\thanks{Center for Applied Mathematics, Cornell University, Ithaca, NY
  (\email{zjf4@cornell.edu}).}
\and Joel A.~Tropp\thanks{Department of Computing and Mathematical Sciences, California Institute of Technology, Pasadena, CA 91125-5000
  (\email{jtropp@cms.edu}).}
\and Madeleine Udell\thanks{Department of Operations Research and Information Engineering, Cornell University, Ithaca, NY(\email{udell@cornell.edu}).}}

\usepackage{amsopn}

\ifpdf
\hypersetup{
  pdftitle={Randomized Nyström Preconditioning},
  pdfauthor={Zachary Frangella, Joel A. Tropp, and Madeleine Udell}
}
\fi

\begin{document}

\maketitle

\begin{abstract}
    This paper introduces the Nystr{\"o}m PCG algorithm
    for solving a symmetric positive-definite linear system.
    The algorithm applies the randomized Nystr{\"o}m method
    to form a low-rank approximation of the matrix,
    which leads to an efficient preconditioner
    that can be deployed with the conjugate gradient algorithm.
    Theoretical analysis
    shows that preconditioned system has constant condition
    number as soon as the rank of the approximation
    is comparable with the number of effective degrees
    of freedom in the matrix.  The paper also develops
    adaptive methods that provably achieve similar performance
    without knowledge of the effective dimension.
    Numerical tests show that Nystr{\"o}m PCG can rapidly solve large linear systems
    that arise in data analysis problems,
    and it surpasses several competing methods from the literature.
\end{abstract}



\section{Motivation}

In their elegant 1997 textbook on numerical linear algebra~\cite{trefethen1997numerical},
Trefethen and Bau write,

\begin{quotation}
``In ending this book with the subject of preconditioners, we find ourselves at
the philosophical center of the scientific computing of the future...
Nothing will be more central to computational science in the next century
than the art of transforming a problem that appears intractable into another
whose solution can be approximated rapidly. For Krylov subspace matrix iterations,
this is preconditioning...
we can only guess where this idea will take us.''
\end{quotation}

The next century has since arrived, and one of the most fruitful developments
in matrix computations has been the emergence of new algorithms that use randomness
in an essential way.  This paper explores a topic at the nexus of preconditioning
and randomized numerical linear algebra.  We will show how to use a
randomized matrix approximation algorithm to construct a preconditioner
for an important class of linear systems that arises throughout
data analysis and scientific computing.

\subsection{The Preconditioner}

Consider the regularized linear system
\begin{equation} \label{eqn:reg-linsys-intro}
(A + \mu I) x = b
\quad\text{where $A \in \R^{n \times n}$ is symmetric psd and $\mu \geq 0$.}
\end{equation}
Here and elsewhere, psd abbreviates the term ``positive semidefinite.''
This type of linear system emerges whenever we solve a regularized
least-squares problem.
We will design a class of preconditioners for the problem~\cref{eqn:reg-linsys-intro}.

Throughout this paper, we assume that we can access the matrix
$A$ through matrix--vector products $x \mapsto Ax$, commonly
known as \textit{matvecs}.  The algorithms that we develop
will economize on the number of matvecs, and they may not be
appropriate in settings where matvecs are very expensive or
there are cheaper ways to interact with the matrix.

For a rank parameter $\ell \in \mathbb{N}$, the randomized Nystr{\"o}m approximation of $A$ takes the form
\begin{equation} \label{eqn:randnys-intro}
\hat{A}_{\textrm{nys}} = (A\Omega)(\Omega^{T}A\Omega)^{\dagger}(A\Omega)^{T}
\quad\text{where $\Omega \in \R^{n \times \ell}$ is standard normal.}
\end{equation}
This matrix provides the best psd approximation of $A$ whose range coincides with the range of the sketch $A \Omega$.  The randomness in the construction ensures that $\hat{A}_{\textrm{nys}}$ is a good approximation to the original matrix $A$ with high probability~\cite[Sec.~14]{MartTroppSurvey}.

We can form the Nystr{\"o}m approximation with sketch size $\ell$, using $\ell$ matvecs with $A$,
plus some extra arithmetic.  See \cref{alg:RandNysAppx} for the implementation details. 

Given the eigenvalue decomposition $\hat{A}_{\textrm{nys}} = U \hat{\Lambda} U^T$
of the randomized Nystr{\"o}m approximation, we construct the Nystr{\"o}m preconditioner:
\begin{equation} \label{eqn:precond-intro}
P = \frac{1}{\hat{\lambda}_{\ell}+\mu}U(\hat{\Lambda}+\mu I)U^{T}+(I-UU^T).
\end{equation}
In a slight abuse of terminology, we refer to $\ell$ as the rank of the Nystr{\"o}m preconditioner.
The key point is that we can solve the linear system $Py = c$ very efficiently,
and the action of $P^{-1}$ dramatically reduces the condition number of the
regularized matrix $A_{\mu} = A + \mu I$.

We propose to use~\cref{eqn:precond-intro} in conjunction with the
preconditioned conjugate gradient (PCG) algorithm.  Each iteration of PCG
involves a single matvec with $A$, and a single linear solve with $P$.
When the preconditioned matrix $P^{-1} A_{\mu}$ has a modest condition
number, the algorithm converges to a solution of~\cref{eqn:reg-linsys-intro}
very quickly.  See~\cref{alg:NysPCG} for pseudocode for Nystr{\"o}m PCG.

The randomized Nystr{\"o}m preconditioner~\cref{eqn:precond-intro} was suggested by P.-G.~Martinsson in the survey~\cite[Sec.~17]{MartTroppSurvey}, but it has not been implemented or analyzed.

\subsection{Guarantees}

  This paper contains the first comprehensive study of the preconditioner~\cref{eqn:precond-intro}, including theoretical analysis and testing on prototypical problems from data analysis and machine learning.  One of the main contributions is a rigorous method for choosing the rank $\ell$ to guarantee good performance, along with an adaptive rank selection procedure that performs well in practice.

A key quantity in our analysis is the \textit{effective dimension} of the regularized matrix $A+\mu I$.  That is, 
\begin{equation}
\label{EffDimDef}
    d_{\textrm{eff}}(\mu) = \tr\left(A(A+\mu I)^{-1}\right) = \sum_{j=1}^{n}\frac{\lambda_j(A)}{\lambda_j(A)+\mu}.\
\end{equation}
The effective dimension measures the degrees of freedom of the problem after regularization. 
It may be viewed as a (smoothed) count of the eigenvalues larger than $\mu$.
Many real-world matrices exhibit strong spectral decay, so the effective dimension is typically much smaller than the nominal dimension $n$.
As we will discuss, the effective dimension also plays a role in a number of machine learning papers~\cite{el2014fast,avron2017faster,bach2013sharp,chowdhury2018randomized,lacotte2020effective} that consider randomized algorithms for solving regularized linear systems.

Our theory tells us the randomized Nystr{\"o}m preconditioner $P$ 
is successful when its rank $\ell$ is proportional to the effective dimension.

\begin{theorem}[Randomized Nystr{\"o}m Preconditioner] 
\label{MainThm}
Let $A\in \mathbb{S}_{n}^+(\mathbb{R})$ be a psd matrix, and write $A_\mu = A+\mu I$ where the regularization parameter $\mu > 0$.  Define the effective dimension $d_{\textup{eff}}(\mu)$ as in~\cref{EffDimDef}.  Construct the randomized preconditioner $P$ from~\cref{eqn:randnys-intro,eqn:precond-intro} with rank parameter $\ell = 2 \,\lceil 1.5 \,d_{\textup{eff}}(\mu)\rceil+1$.  Then the condition number of the preconditioned system satisfies
\begin{equation} \label{eqn:expect-cond-intro}
\E\big[ \kappa_{2}(P^{-1/2}A_\mu P^{-1/2}) \big] < 28.
\end{equation}
\end{theorem}
\noindent
\Cref{MainThm} is a restatement of \cref{NysPCGEffDimThm}.

Simple probability bounds follow from~\cref{eqn:expect-cond-intro} via Markov's inequality.  For example,
\[
\P\big\{ \kappa_{2}(P^{-1/2}A_\mu P^{-1/2}) \leq 56 \big\} > 1/2.
\]
The main consequence of \cref{MainThm} is a convergence theorem for PCG with the randomized Nystr{\"o}m preconditioner.

\begin{corollary}[Nystr{\"o}m PCG: Convergence]
Construct the preconditioner $P$ as in \cref{MainThm}, and condition on the event 
$\{ \kappa_2(P^{-1/2}A_\mu P^{-1/2}) \leq 56 \}$.
Solve the regularized linear system~\cref{eqn:reg-linsys-intro} using Nystr{\"o}m PCG, starting with an initial iterate $x_0 = 0$.  After $t$ iterations, the relative error $\delta_t$ satisfies
\[
\delta_t \coloneqq \frac{\|x_{t}-x_{\star}\|_{\rm PCG}}{\|x_{\star}\|_{\rm PCG}}< 2 \cdot \left(0.77 \right)^{t}
\quad\text{where $A_{\mu} x_{\star} = b$.}
\]
The error norm is defined as $\| u \|_{\rm PCG}^2 = u^T (P^{-1/2} A_{\mu} P^{-1/2} ) u$.
In particular, $t\geq \lceil 3.9 \log(2/\epsilon) \rceil$ iterations
suffice to achieve relative error $\epsilon$.
\end{corollary}

Although \cref{MainThm} gives an interpretable bound for the rank $\ell$ of the preconditioner,
we cannot instantiate it without knowledge of the effective dimension.  To address this shortcoming,
we have designed adaptive methods for selecting the rank in practice (\cref{sec:adaptive-rank}).

Finally,  as part of our investigation, we will also develop a detailed understanding
of Nystr{\"o}m sketch-and-solve, a popular algorithm in
the machine learning literature \cite{el2014fast,bach2013sharp}.
Our analysis highlights the deficiencies of Nystr{\"o}m sketch-and-solve relative to Nystr{\"o}m PCG.

\subsection{Example: Ridge Regression}
\label{IntroRidgeEx}

As a concrete example, we consider the $\ell^2$ regularized least-squares problem, also known as ridge regression.
This problem takes the form
\begin{equation}\label{eq:ridge-intro}
\mathrm{minimize}_{x\in \mathbb{R}^{d}} \quad \frac{1}{2n}\|Gx-b\|^{2}+\frac{\mu}{2}\|x\|^{2},
\end{equation}
where $G\in \R^{n\times d}$ and $b\in \R^{n}$ and $\mu>0$.
By calculus, the solution to \cref{eq:ridge-intro} also satisfies the regularized system of linear equations
\begin{equation} 
 \label{eq-RidgeRegression}
 \left(\frac{1}{n}G^{T}G+\mu I\right)x = \frac{1}{n}G^{T}b.   
\end{equation}
A direct method to solve \cref{eq-RidgeRegression} requires $O(nd^{2})$ flops,
which is prohibitive when $n$ and $d$ are both large. 
Instead, when $n$ and $d$ are large, iterative algorithms, such as the conjugate gradient method (CG), become the tools of choice.
Unfortunately, the ridge regression linear system~\cref{eq-RidgeRegression} is often very ill-conditioned, and CG converges very slowly. 

Nystr{\"o}m PCG can dramatically accelerate the solution of \cref{eq-RidgeRegression}. 
As an example, consider the {shuttle-rf} dataset (\cref{section:RidgeRegression}).  The matrix $G$ has dimension $43,300 \times 10,000$, while the preconditioner is based on a Nystr{\"o}m approximation with rank $\ell = 800$.
\Cref{fig:shuttlerfDemo} shows the progress of the residual as a function of the iteration count.  Nystr{\"o}m PCG converges to machine precision in 13 iterations, while CG stalls.

\begin{figure}[t]
    \centering
    \includegraphics[width=0.65\textwidth]{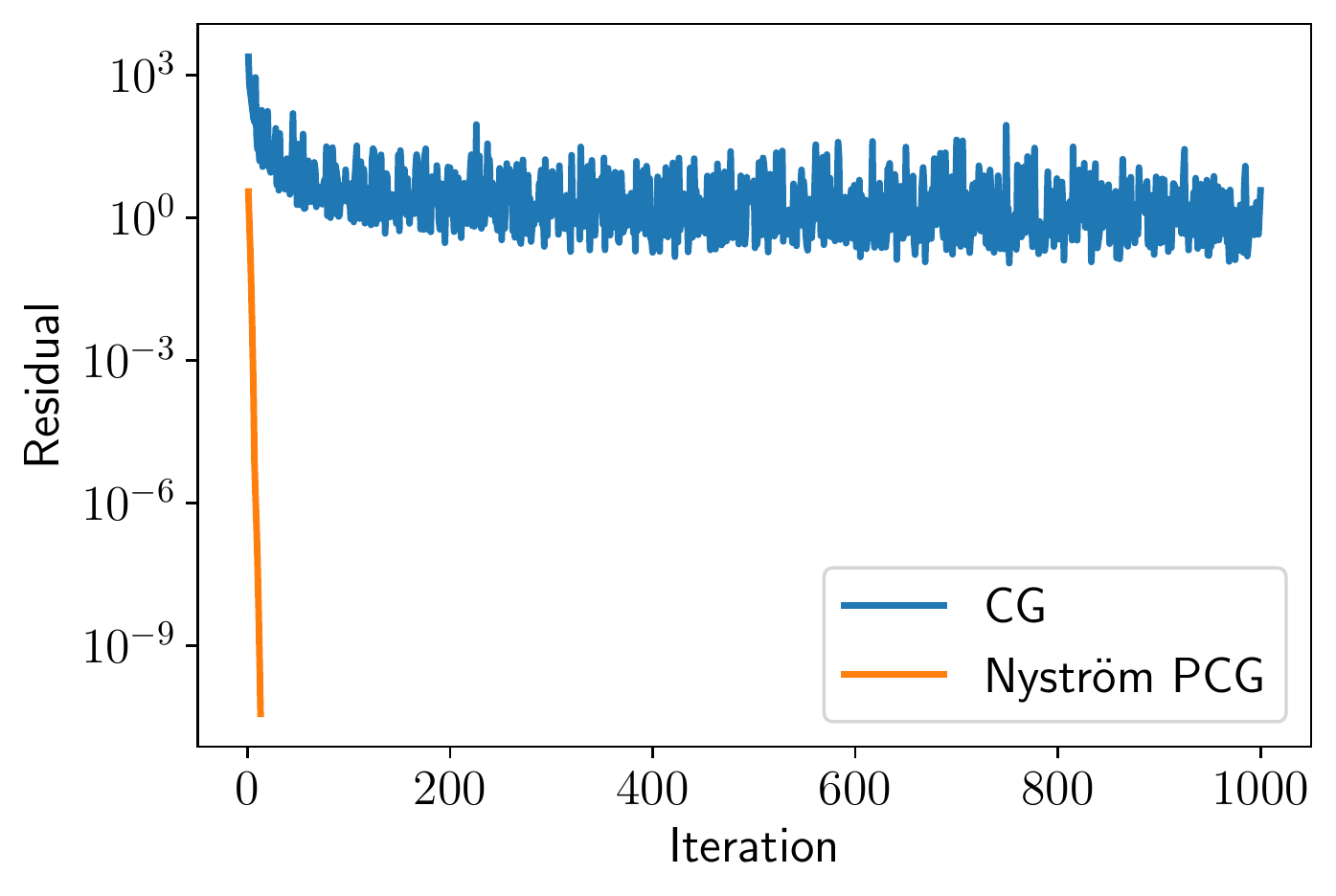}
    \caption{\label{fig:shuttlerfDemo}
    \textbf{Ridge regression: CG versus Nystr{\"o}m PCG.}
    For the {shuttle-rf} data set, Nystr{\"o}m PCG converges to machine precision in 13 iterations while CG stalls.  See~\cref{IntroRidgeEx,section:RidgeRegression}.}
\end{figure}

\subsection{Comparison to prior randomized preconditioners}
Prior proposals for randomized preconditioners \cite{avron2010blendenpik, meng2014lsrn, rokhlin2008fast}
accelerate the solution of highly overdetermined or underdetermined least-squares problems using the sketch-and-precondition paradigm \cite[Sec.~10]{MartTroppSurvey}.
For $n\geq d$, these methods require $\Omega(d^{3})$ 
computation to factor the preconditioner.

In contrast, the randomized Nystr{\"o}m preconditioner 
applies to any symmetric positive-definite linear system
and can be significantly faster for regularized problems. 
See \cref{sec:preconditioner-comparison} more details.

\subsection{Roadmap}
\Cref{section:NystromApproximation} contains an overview of the Nystr{\"o}m approximation and its key properties.  \Cref{section:AppxRegInverse} studies the role of the Nystr{\"o}m approximation in estimating the inverse of the regularized matrix.  We analyze the Nystr{\"o}m sketch-and-solve method in \Cref{Section:NystromSketchSolve}, and we give a rigorous performance bound for this algorithm.  \Cref{section:NystromPCG} presents a full treatment of Nystr{\"o}m PCG, including theoretical results and guidance on numerical implementation.  Computational experiments in \Cref{section:NumericalExperiments} demonstrate the power of Nystr{\"o}m PCG for three different applications involving real data sets.

\subsection{Notation}

We write $\mathbb{S}_{n}(\R)$ for the linear space of $n\times n$ real symmetric matrices, while $\mathbb{S}_n^+(\R)$ denotes the convex cone of real psd matrices. The symbol $\preceq$ denotes the Loewner order on $\mathbb{S}_{n}(\R)$.  That is, $A \preceq B$ if and only if the eigenvalues of $B-A$ are all nonnegative.
The function $\mathrm{tr}[\cdot]$ returns the trace of a square matrix.  The map $\lambda_j(A)$ returns the $j$th largest eigenvalue of $A$; we may omit the matrix if it is clear.  As usual, $\kappa_2$ denotes the $\ell^2$ condition number.
We write $\|M\|$ for the spectral norm of a matrix $M$.  For a psd matrix $A$, we write $\| u \|_{A}^2 = u^T A u$ for the $A$-norm.
Given $A\in \mathbb{S}_{n}(\R)$ and $1\leq \ell\leq n$, the symbol $\lfloor A\rfloor_{\ell}$ refers to any best rank-$\ell$ approximation to $A$ relative to the spectral norm.
For $A \in \mathbb{S}_{n}^+(\R)$ and $\mu \geq 0$, the regularized matrix is abbreviated $A_{\mu} = A + \mu I$. 
For $A\in \mathbb{S}_{n}^{+}(\R)$ and $\mu>0$ effective dimension of $A_{\mu}$ is defined as $d_{\textrm{eff}}(\mu) = \tr(A(A+\mu I))^{-1}$.
For $A\in \mathbb{S}_{n}^{+}(\R)$, the $p$-stable rank of $A$ is defined as $\textrm{sr}_{p}(A) = \lambda_{p}^{-1}\sum^{n}_{j>p}\lambda_{j}$.
For $A\in \mathbb{S}_{n}^{+}(\R)$, we denote the time taken to compute a matvec with $A$ by $T_{\textrm{mv}}$.
\section{The Nystr{\"o}m approximation}
\label{section:NystromApproximation}

Let us begin with a review of the Nystr{\"o}m approximation and the randomized Nystr{\"o}m approximation. 

\subsection{Definition and basic properties}

The Nystr{\"o}m approximation is a natural way to construct a low-rank psd approximation of a psd matrix $A\in \mathbb{S}_{n}^+(\R)$. Let $X\in \R^{n\times \ell}$ be an arbitrary test matrix.  The \textit{Nystr{\"o}m approximation} of $A$ with respect to the range of $X$ is defined by
\begin{equation}
    \label{NysApprox}
    A\langle X \rangle = (AX)(X^{T}AX)^{\dagger}(AX)^{T} \in \mathbb{S}_n^+(\R).
\end{equation}
The Nystr{\"o}m approximation is the best psd approximation of $A$
whose range coincides with the range of $AX$.  It has a deep relationship with the Schur complement and with Cholesky factorization~\cite[Sec.~14]{MartTroppSurvey}.

The Nystr{\"o}m approximation enjoys several elementary properties that we record in the following lemma. 

\begin{lemma}
\label{NysPropLemma}
Let $A\langle X\rangle \in \mathbb{S}_n^+(\R)$ be a Nystr{\"o}m approximation of the psd matrix $A \in \mathbb{S}_{n}^+(\mathbb{R})$.
Then 
\begin{enumerate}
    \item The approximation $A\langle X\rangle$ is psd and has rank at most $\ell$.
    \item The approximation $A\langle X \rangle$ depends only on $\mathrm{range}(X)$.
    \item\label{it:nys-order} In the Loewner order, $A\langle X\rangle\preceq A$.
    \item\label{it:nys-eigs} In particular, the eigenvalues satisfy $\lambda_j(\hat{A}) \leq \lambda_{j}(A)$ for each $1 \leq j \leq n$.
\end{enumerate}
\end{lemma}
The proof of \cref{NysPropLemma}, \cref{it:nys-order} is not completely obvious.  It is a consequence of the fact that we may express $\hat{A}_{\textrm{nys}} = A^{1/2}\Pi A^{1/2}$, where $\Pi$ is an orthogonal projector.

\subsection{Randomized Nystr{\"o}m approximation}
\label{section:RandNysAppx}
\begin{algorithm}[t]
	\caption{Randomized Nystr{\"o}m Approximation \cite{li2017algorithm,tropp2017fixed}}
	\label{alg:RandNysAppx}
	\begin{algorithmic}[1] 
		\Require{Positive-semidefinite matrix $A \in \mathbb{S}_n^+(\R)$, rank $\ell$}
		\Ensure{ Nystr{\"o}m approximation in factored form $\hat{A}_{\mathrm{nys}} = U \hat{\Lambda} U^T$}
		
		\vspace{0.5pc}

		\State $\Omega = \textrm{randn}(n, \ell)$
		    \Comment{Gaussian test matrix}
		\State $\Omega = \textrm{qr}(\Omega,0)$
		    \Comment{Thin QR decomposition}
		\State $Y = A\Omega$
		    \Comment{$\ell$ matvecs with $A$}
		\State $\nu = \textrm{eps}(\textrm{norm}(Y,\textrm{'fro'}))$
		    \Comment{Compute shift}
		\State $Y_{\nu} = Y+\nu\Omega$
		    \Comment{Shift for stability}
		\State $C = \textrm{chol}(\Omega^{T}Y_{\nu})$
		\State $B = Y_{\nu}/C$
		\State $[U,\Sigma,\sim] = \textrm{svd}(B,0)$
		    \Comment{Thin SVD}
		\State $\hat{\Lambda} = \max\{0,\Sigma^{2}-\nu I\}$ \Comment{Remove shift, compute eigs}
	\end{algorithmic}
\end{algorithm}

How should we choose the test matrix $X$ so that the Nystr{\"o}m approximation $A\langle X\rangle$ provides a good low-rank model for $A$?  Surprisingly, we can obtain a good approximation simply by drawing the test matrix at random.
See~\cite{tropp2017fixed} for theoretical justification of this claim.

Let us outline the construction of the randomized Nystr{\"o}m approximation.
Draw a standard normal test matrix $\Omega \in \R^{n \times \ell}$ where $\ell$ 
is the sketch size, and compute the sketch $Y = A\Omega$ By \cref{NysPropLemma}, the sketch size $\ell$ is equal to the rank of $\hat{A}_{\textrm{nys}}$ with probability 1,
hence we use these terms interchangeably.
The Nystr{\"o}m approximation~\cref{NysApprox} is constructed 
directly from the test matrix $\Omega$ and the sketch $Y$:
\begin{equation} \label{eqn:nys-unstable}
    \hat{A}_{\textrm{nys}} = A\langle\Omega\rangle = Y(\Omega^{T}Y)^{\dagger}Y^{T}.
\end{equation}
The formula~\cref{eqn:nys-unstable} is not numerically sound.
We refer the reader to~\cref{alg:RandNysAppx}
for a stable and efficient implementation of the
randomized Nystr{\"o}m approximation~\cite{li2017algorithm,tropp2017fixed}.
Conveniently, \cref{alg:RandNysAppx} returns the truncated eigendecomposition $\hat{A}_{\textrm{nys}} = U\hat{\Lambda}U^{T}$, where $U \in \mathbb{R}^{n\times \ell}$ is an orthonormal matrix whose columns
are eigenvectors and  $\hat{\Lambda}\in \R^{\ell\times \ell}$
is a diagonal matrix listing the eigenvalues, which we often abbreviate
as $\hat{\lambda}_1, \dots, \hat{\lambda}_{\ell}$.

The randomized Nystr{\"o}m approximation described in this section
has a key difference from the Nystr{\"o}m approximations that
have traditionally been used in the machine learning literature~\cite{el2014fast,bach2013sharp,derezinski2020improved,gittens2011spectral,williams2001using}.
In machine learning settings, the Nystr{\"o}m approximation is usually constructed from a sketch $Y$ that samples random columns from the matrix (i.e., the random test matrix $\Omega$ has 1-sparse columns).  In contrast, \cref{alg:RandNysAppx} computes a sketch $Y$ via random projection (i.e., the test matrix $\Omega$ is standard normal).  In most applications, we have strong reasons (\cref{RandProjvsColSamp}) for preferring random projections to column sampling.

\subsubsection{Cost of randomized Nystr{\"o}m approximation} 
\label{section:NystromAppxCost}
 Throughout the paper, we write $T_{\textrm{mv}}$ for the time required to compute a matrix--vector product (matvec) with $A$. Forming the sketch $Y = A\Omega$ with sketch size $\ell$ requires $\ell$ matvecs, which costs $T_{\textrm{mv}}\ell$. The other steps in the algorithm have arithmetic cost $O(n\ell^{2})$.  Hence, the total computational cost of \cref{alg:RandNysAppx} is $O( T_{\textrm{mv}}\ell + \ell^{2} n )$ operations. The storage cost is $O(\ell n)$ floating-point numbers. 

For \cref{alg:RandNysAppx}, the worst-case performance occurs when $A$ is dense and unstructured.  In this case, forming $Y$ costs $O(n^{2}\ell)$ operations.  However, if we have access to the columns of $A$ then we may reduce the cost of forming $Y$ to $O(n^{2}\log \ell)$ by using a structured test matrix $\Omega$, such as a scrambled subsampled randomized Fourier transform (SSRFT) map or a sparse map~\cite{MartTroppSurvey,tropp2017fixed}.

\subsubsection{A priori guarantees for the randomized Nystr{\"o}m approximation}

In this section, we present an a priori error bound for the randomized Nystr{\"o}m approximation.
The result improves over previous analyses~\cite{gittens2011spectral,gittens2013revisiting,tropp2017fixed}
by sharpening the error terms.  This refinement is critical for the analysis of the preconditioner.
\begin{proposition}[Randomized Nystr{\"o}m approximation: Error]  
\label{NysExpBound} 
Consider a psd matrix $A\in \mathbb{S}_{n}^+(\R)$ with 
eigenvalues $\lambda_{1} \geq \cdots \geq \lambda_{n}$. 
Choose a sketch size $\ell \geq 4$, and draw a standard normal test matrix $\Omega \in \R^{n\times \ell}$. 
Then the rank-$\ell$ Nystr{\"o}m approximation $\hat{A}_{\textup{nys}}$ computed by \cref{alg:RandNysAppx} satisfies
\begin{equation}
\label{eq-RandNysAppxErrBnd}
    \mathbb{E}\|A-\hat{A}_{\textup{nys}}\| \leq \min_{2 \leq p \leq \ell-2} \left[ \left(1+\frac{2(\ell-p)}{p-1}\right)\lambda_{\ell-p+1}
    +\frac{2\textup{e}^{2}\ell}{p^{2}-1}\left(\sum_{j>\ell - p}\lambda_{j}\right) \right].
\end{equation}
\end{proposition}

\noindent
The proof of \cref{NysExpBound} may be found in \cref{Section:NysExpBoundPf}. 

\cref{NysExpBound} shows that, in expectation, the randomized Nystr{\"o}m approximation $\hat{A}_{\textrm{nys}}$ provides a good rank-$\ell$ approximation to $A$.  The first term in the bound is comparable with the spectral-norm error $\lambda_{\ell - p + 1}$ in the optimal rank-$(\ell - p)$ approximation, $\lfloor A \rfloor_{\ell - p}$.  The second term in the bound is comparable with the trace-norm error $\sum_{j > \ell - p} \lambda_j$ in the optimal rank-$(\ell - p)$ approximation.

\Cref{NysExpBound} is better understood via the following simplification.

\begin{corollary}[Randomized Nystr{\"o}m approximation]
 \label{NysExpCorr}
Instate the assumptions of~\cref{NysExpBound}.  For $p \geq 2$ and $\ell = 2p - 1$, we have the bound 
\[\mathbb{E}\|A-\hat{A}_{\textup{nys}}\| \leq \left(3+\frac{4\textup{e}^{2}}{p}\textup{sr}_{p}(A)\right)\lambda_{p}.\]
The $p$-stable rank, $\textup{sr}_{p}(A) = \lambda_{p}^{-1}\sum^{n}_{j = p}\lambda_{j}$, reflects decay in the tail eigenvalues.
\end{corollary}

\cref{NysExpCorr} shows that the Nystr{\"o}m approximation error is on the order of $\lambda_{p}$ when the rank parameter $\ell = 2p - 1$.  The constant depends on the $p$-stable rank $\textrm{sr}_p(A)$, which is small when the tail eigenvalues decay quickly starting at $\lambda_p$.  This bound is critical for establishing our main results (\cref{NysInvBnd,NysPCGEffDimThm}).   

\subsubsection{Random projection versus column sampling}
\label{RandProjvsColSamp}

Most papers in the machine learning literature~\cite{el2014fast,bach2013sharp} construct Nystr{\"o}m approximations by sampling columns at random from an adaptive distribution.
In contrast, for most applications, we advocate using an oblivious random projection of the matrix to construct a Nystr{\"o}m approximation.

Random projection has several advantages over column sampling. 
First, column sampling may not be practical when we only have black-box matvec access to the matrix,
while random projections are natural in this setting.
Second, it can be very expensive to obtain adaptive distributions for column sampling.  Indeed, computing approximate ridge leverage scores costs just as much as solving the ridge regression problem directly using random projections \cite[Theorem 2]{drineas2012fast}. 
Third, even with a good sampling distribution, column sampling
produces higher variance results than random projection,
so it is far less reliable.

On the other hand, we have found that there are a few applications where it is more effective to compute a randomized Nystr{\"o}m preconditioner using column sampling in lieu of random projections.  In particular, this seems to be the case for kernel ridge regression (\cref{section:KernelRidgeRegression}).  Indeed, the entries of the kernel matrix are given by an explicit formula, so we can extract full columns with ease.  Sampling $\ell$ columns may cost only $O(\ell n)$ operations, whereas a single matvec generally costs $O(n^2)$.  Furthermore, kernel matrices usually exhibit fast spectral decay, which limits the performance loss that results from using column sampling in lieu of random projection.

\section{Approximating the regularized inverse}
\label{section:AppxRegInverse}

Let us return to the regularized linear system~\cref{eqn:reg-linsys-intro}.
The solution to the problem has the form $x_{\star} = (A+\mu I)^{-1}b$. 
Given a good approximation $\hat{A}$ to the matrix $A$, it is natural to ask whether $\hat{x} = (\hat{A} + \mu I)^{-1} b$ is a good approximation to the desired solution $x_{\star}$.

There are many reasons why we might prefer to use $\hat{A}$ in place of $A$.  In particular, we may be able to solve linear systems in the matrix $\hat{A} + \mu I$ more efficiently.  On the other hand, the utility of this approach depends on how well the inverse $(\hat{A} + \mu I)^{-1}$ approximates the desired inverse $(A + \mu I)^{-1}$.  The next result addresses this question for a wide class of approximations that includes the Nystr{\"o}m approximation. 



\begin{proposition}[Regularized inverses]
\label{GenInvErrBnd}
Consider psd matrices $A$, $\hat{A}\in \mathbb{S}_n^+(\R)$, and assume that the difference $E = A-\hat{A}$ is psd.  Fix $\mu > 0$.  Then
\begin{equation}
    \label{eq-InvInequality}
    \|(\hat{A}+\mu I)^{-1}-(A+\mu I)^{-1}\|\leq \frac{1}{\mu}\frac{\|E\|}{\|E\|+\mu}.
\end{equation}
Furthermore, the bound~\eqref{eq-InvInequality} is attained when $\hat{A} = \lfloor A\rfloor_{\ell}$ for $1\leq \ell \leq n$.
\end{proposition}
The proof of \cref{GenInvErrBnd} may be found in \cref{section:GenInvErrBndPf}.  It is based on~\cite[Lemma~X.1.4]{bhatia2013matrix}.


\cref{GenInvErrBnd} has an appealing interpretation.  When $\|A - \hat{A}\|$ is small in comparison to the regularization parameter $\mu$, then the approximate inverse $(\hat{A} + \mu I)^{-1}$ can serve in place of the inverse $(A + \mu I)^{-1}$.  Note that $\|(A + \mu I)^{-1}\| \leq 1/\mu$, so we can view~\cref{eq-InvInequality} as a relative error bound.

\section{Nystr{\"o}m sketch-and-solve}

The simplest mechanism for using the Nystr{\"o}m approximation
is an algorithm called Nystr{\"o}m sketch-and-solve.  This
section introduces the method, its implementation, and its history.
We also provide a general theoretical analysis that sheds light on its performance.  In spite of its popularity, the Nystr{\"o}m sketch-and-solve method is rarely worth serious consideration.

\begin{algorithm}[t]
	\caption{Nystr{\"o}m sketch-and-solve} 
	\label{alg:NysSketchSolve}
	\begin{algorithmic}[1] 
		\Require{Psd matrix $A \in \mathbb{S}_n^+(\R)$, right-hand side $b$, regularization $\mu$, rank $\ell$} 
		\Ensure{Approximate solution $\hat{x}$ to~\cref{eqn:reg-linsys-intro}}
\vspace{0.5pc}
		\State $[U,\hat{\Lambda}]$ = RandomizedNystr{\"o}mApproximation($A,\ell$)
		\State Use~\cref{eq-Woodbury} to compute $\hat{x} = (\hat{A}_{\textrm{nys}}+\mu I)^{-1}b$
	\end{algorithmic}
\end{algorithm}

\label{Section:NystromSketchSolve}
\subsection{Overview}

Given a rank-$\ell$ Nystr{\"o}m approximation $\hat A_\textrm{nys}$ of the psd matrix $A$, it is tempting to replace the regularized linear system $(A + \mu I) x = b$ with the proxy $(\hat{A}_{\textrm{nys}} + \mu I) x = b$.  Indeed, we can solve the proxy linear system in $O(\ell n)$ time using the Sherman--Morrison--Woodbury formula~\cite[Eqn. (2.1.4)]{golubvanloan2013}:

\begin{lemma}[Approximate regularized inversion]
\label{SketchInv}
Consider any rank-$\ell$ matrix $\hat{A}$ with eigenvalue decomposition $\hat{A} = U \hat \Lambda U^T$.
Then
\begin{equation}
\label{eq-Woodbury}
    (\hat{A}+\mu I)^{-1} = U(\hat{\Lambda}+\mu I)^{-1}U^{T}+\frac{1}{\mu}(I-UU^{T}).
\end{equation}
\end{lemma}

\noindent
We refer to the approach in this paragraph as the Nystr{\"o}m sketch-and-solve algorithm
because it is modeled on the sketch-and-solve paradigm that originated in~\cite{sarlos2006improved}.

See~\cref{alg:NysSketchSolve} for a summary of the Nystr{\"o}m sketch-and-solve method.
The algorithm produces an approximate solution $\hat x$ to the regularized linear system~\cref{eqn:reg-linsys-intro} in time $O(T_{\textrm{mv}}\ell+\ell^{2} n)$.  The arithmetic cost is much faster than a direct method, which costs $O(n^3)$.  It can also be faster than running CG for a long time at a cost of $O(T_{\textrm{mv}})$ per iteration.  The method looks attractive if we only consider the runtime, and yet...

Nystr{\"o}m sketch-and-solve only has one parameter, the rank $\ell$ of the Nystr{\"o}m approximation, which controls the quality of the approximate solution $\hat x$. When $\ell \ll n$, the method has an appealing computational profile.  As $\ell$ increases, the approximation quality increases but the computational burden becomes heavy.  We will show that, alas, an accurate solution to the linear system actually requires $\ell\approx n$, at which point the computational benefits of Nystr{\"o}m sketch-and-solve evaporate completely.

In summary, Nystr{\"o}m sketch-and-solve is almost never the right algorithm to use.  We will see that Nystr{\"o}m PCG generally produces much more accurate solutions with a similar computational cost.

\subsection{Guarantees and deficiencies}
\label{NysSketchSolveAnalysis}

Using \cref{GenInvErrBnd} together with the a priori guarantee in \cref{NysExpBound},  we quickly obtain a performance guarantee for \cref{alg:NysSketchSolve}. 

\begin{theorem}
\label{NysInvBnd}
Fix $p\geq 2$, and set $\ell = 2p - 1$.  For a psd matrix $A \in \mathbb{S}_n^+(\R)$, construct a randomized Nystr{\"o}m approximation $\hat{A}_{\textup{nys}}$ using \cref{alg:RandNysAppx}. 
Then the approximation error for the inverse satisfies
\begin{equation}
     \label{eq-NysInvExpErr}
     \mathbb{E}\big\|(A+\mu I)^{-1}-(\hat{A}_{\textup{nys}}+\mu I)^{-1}\big\|\leq \bigg(3+\frac{4\textup{e}^{2}}{p}\textup{sr}_{p}(A)\bigg)\frac{\lambda_{p}}{\mu \cdot (\lambda_{p}+\mu)}.
\end{equation}
Define $x_{\star} = (A+\mu I)^{-1}b$, and select $\ell = 2\,\lceil 1.5 \, d_{\textup{eff}}(\epsilon\mu)\rceil+1$.
Then the approximate solution $\hat{x}$ computed by \cref{alg:NysSketchSolve} satisfies
\begin{equation}
    \label{eq-NysSSRelErr}
    \E\bigg[\frac{\|\hat{x}-x_{\star}\|_{2}}{\|x_{\star}\|_{2}}\bigg] \leq 26\epsilon. 
\end{equation}
\end{theorem}
The proof of \cref{NysInvBnd} may be found in \cref{section:NysSketchSolveProof}.

\cref{NysInvBnd} tells us how accurately we can hope to solve linear systems using Nystr{\"o}m sketch-and-solve (\cref{alg:NysSketchSolve}). 
To obtain relative error $\epsilon$, we should choose $\ell = O(d_{\textrm{eff}}(\epsilon\mu))$.  When $\epsilon \mu$ is small, we anticipate that $d_{\textrm{eff}}(\epsilon\mu)\approx n$.  In this case, Nystr{\"o}m sketch-and-solve has no computational value.  Our analysis is sharp in its essential respects, so the pessimistic assessment is irremediable.

\subsection{History}

Nystr{\"o}m sketch-and-solve has a long history in the machine learning literature.  It was introduced in \cite{williams2001using} to speed up kernel-based learning, and it plays a role in many subsequent papers on kernel methods.  In this context, the Nystr{\"o}m approximation is typically obtained using column sampling \cite{el2014fast,bach2013sharp,williams2001using}, which has its limitations (\cref{RandProjvsColSamp}).  More recently, Nystr{\"o}m sketch-and-solve has been applied to speed up approximate cross-validation \cite{stephenson2020LowRank}.

The analysis of Nystr{\"o}m sketch-and-solve presented above differs from previous analysis. Prior works~\cite{el2014fast,bach2013sharp} focus on the kernel setting, and they use properties of column sampling schemes to derive learning guarantees.  In contrast, we bound the relative error for a Nystr{\"om} approximation based on a random projection.
Our overall approach extends to column sampling if we replace \cref{NysExpBound} by an appropriate analog, such as Gittens's results~\cite{gittens2011spectral}.

\section{Nystr{\"o}m Preconditioned Conjugate Gradients}
\label{section:NystromPCG}
We now present our main algorithm, Nystr{\"o}m PCG.  This algorithm produces high accuracy solutions to a regularized linear system by using the Nystr{\"o}m approximation $\hat{A}_{\textrm{nys}}$ as a preconditioner.  We provide a rigorous estimate for the condition number of the preconditioned system, and we prove that Nystr{\"o}m PCG leads to fast convergence for regularized linear systems.  In contrast, we have shown that Nystr{\"o}m sketch-and-solve cannot be expected to yield accurate solutions.
  
\subsection{The preconditioner}

In this section, we introduce the optimal low-rank preconditioner,
and we argue that the randomized Nystr{\"o}m preconditioner provides an approximation that is easy to compute.

\subsubsection{Motivation}
As a warmup, suppose we knew the eigenvalue decomposition of the best rank-$\ell$ approximation of the matrix: $\lfloor A \rfloor_\ell = V_{\ell} \Lambda_{\ell} V_{\ell}^T$.
How should we use this information to construct a good preconditioner for the regularized linear system~\cref{eqn:reg-linsys-intro}?

Consider the family of symmetric psd matrices that act as the identity on the orthogonal complement of $\mathrm{range}(V_{\ell})$.  Within this class, we claim that the following matrix is the \textit{optimal preconditioner}:
\begin{equation}
\label{eq-OptLowRankPre} 
    P_\star = \frac{1}{\lambda_{\ell+1}+\mu}V_{\ell}(\Lambda_{\ell}+\mu I)V_{\ell}^{T}+(I-V_{\ell}V_{\ell}^{T}). 
\end{equation}
The optimal preconditioner $P_{\star}$ requires $O(n\ell)$ storage, and we can solve linear systems in $P_{\star}$ in $O(n\ell)$ time.  Whereas the regularized matrix $A_{\mu}$ has condition number $\kappa_2(A_{\mu}) = (\lambda_{1}+\mu)/(\lambda_{n}+\mu)$, the preconditioner yields
\begin{equation}
\label{eq-OptCondNumber}
\kappa_{2}(P_\star^{-1/2}A_{\mu}P^{-1/2}_{\star}) = \frac{\lambda_{\ell+1}+\mu}{\lambda_{n}+\mu}.
\end{equation}
This is the minimum possible condition number attainable by a preconditioner  from the class that we have delineated.  It represents a significant improvement when $\lambda_{\ell + 1} \ll \lambda_1$.
The proofs of these claims are straightforward; for details, see \cref{section:PrecondPfs}.

\subsubsection{Randomized Nystr{\"o}m preconditioner}

It is expensive to compute the best rank-$\ell$ approximation $\lfloor A \rfloor_\ell$ accurately.  In contrast, we can compute the rank-$\ell$ randomized Nystr{\"o}m approximation $\hat{A}_{\mathrm{nys}}$ efficiently (\cref{alg:RandNysAppx}).
Furthermore, we have seen that $\hat{A}_{\mathrm{nys}}$ approximates $A$ nearly as well as the optimal rank-$\ell$ approximation (\cref{NysExpCorr}).
These facts lead us to study the randomized Nystr{\"o}m preconditioner,
proposed in~\cite[Sec.~17]{MartTroppSurvey} without a complete justification.

Consider the eigenvalue decomposition $\hat{A}_{\mathrm{nys}} = U \hat{\Lambda} U^T$, and write $\hat{\lambda}_{\ell}$ for its $\ell$th eigenvalue.  The randomized Nystr{\"o}m preconditioner and its inverse take the form
\begin{equation}
\label{eq-NysPre}
\begin{aligned}
    P &= \frac{1}{\hat{\lambda}_{\ell}+\mu}U(\hat{\Lambda}+\mu I)U^{T}+(I-UU^{T}); \\
    P^{-1} &= (\hat{\lambda}_{\ell}+\mu)U(\hat{\Lambda}+\mu I)^{-1}U^{T}+(I-UU^{T}).
\end{aligned}
\end{equation}
Like the optimal preconditioner $P_\star$, the randomized Nystr{\"o}m preconditioner \eqref{eq-NysPre} is cheap to apply and to store.
We may hope that it damps the condition number of the preconditioned system $P^{-1/2}A_{\mu}P^{-1/2}$ nearly as well as the optimal preconditioner $P_\star$.  We will support this intuition with a rigorous bound (\cref{NysPCGThm}).

\subsection{Nystr{\"o}m PCG}

We can obviously use the randomized Nystr{\"o}m preconditioner within the framework of PCG.
We call this approach Nystr{\"o}m PCG, and we present a basic implementation in \cref{alg:NysPCG}.

More precisely, \cref{alg:NysPCG} uses left-preconditioned CG.  This algorithm implicitly works with the unsymmetric matrix $P^{-1}A_{\mu}$, rather than the symmetric matrix $P^{-1/2}{A_\mu}P^{-1/2}$.  The two methods yield identical sequences of iterates~\cite{saad2003iterative}, but the former is more efficient.  For ease of analysis, our theoretical results are presented in terms of the symmetrically preconditioned matrix.

\begin{algorithm}[t]
	\caption{Nystr{\"o}m PCG} 
	\label{alg:NysPCG}
	\begin{algorithmic}[1] 
	\Require{Psd matrix $A$, righthand side $b$, initial guess $x_{0}$, 
	regularization parameter $\mu$, sketch size $\ell$,
	solution tolerance $\eta$}
    \Ensure{Approximate solution $\hat{x}$ to regularized system~\cref{eqn:reg-linsys-intro}}
\vspace{0.5pc}
    \State $[U, \hat{\Lambda}] =$ RandomizedNystr{\"o}mApproximation$(A, \ell)$
	\State $r_{0} = b-(A+ \mu I) x_{0}$
	\State $z_{0} = P^{-1}r_{0}$ \Comment{using \eqref{eq-NysPre}}
	\State $p_{0} = z_{0}$
		\While {$\|r\|_{2}>\eta$}
		\State $v = (A+ \mu I) p_{0}$ 
		\State $\alpha = (r_0^{T}z_{0}) / (p_{0}^{T}v_{0})$ \Comment{compute step size} 
		\State $x = x_{0}+\alpha p_{0}$ \Comment{update solution}
		\State $r = r_{0}-\alpha v$ \Comment{update residual}
		\State $z = P^{-1}r$ \Comment{find search direction via \eqref{eq-NysPre}}
		\State $\beta = (r^{T}z)/(r_{0}^{T}z_{0})$
		\State $x_{0}\leftarrow x$, $r_{0} \leftarrow r$, $p_{0}\leftarrow z+\beta p_{0}$, $z_{0}\leftarrow z$ 
		\EndWhile
	\end{algorithmic}
\end{algorithm}

\subsubsection{Complexity of Nystr{\"o}m PCG}

Nystr{\"o}m PCG has two steps.  First, we construct the randomized Nystr{\"o}m approximation, and then we solve the regularized linear system using PCG.
We have already discussed the cost of constructing the Nystr{\"o}m approximation (\cref{section:NystromAppxCost}).
The PCG stage costs $O(T_{\max})$ operations per iteration, and it uses a total of $O(n)$ additional storage.

For the regularized linear system~\cref{eqn:reg-linsys-intro}, \cref{NysPCGEffDimThm,ConvergenceCorollary} demonstrate that it is enough to choose the sketch size $\ell = 2 \, \lceil 1.5 d_{\mathrm{eff}}(\mu) \rceil + 1$.  In this case, the overall runtime of Nystr{\"o}m PCG is
$$
O\big(d_{\mathrm{eff}}(\mu)^2 n +  T_{\mathrm{mv}} ( d_{\mathrm{eff}}(\mu) + \log(1/\epsilon) ) \big)
\quad\text{operations.}
$$
When the effective dimension $d_{\mathrm{eff}}(\mu)$ is modest, Nystr{\"o}m PCG is very efficient.

In contrast, \cref{NysSketchSolveAnalysis} shows that the running time for Nystr{\"o}m sketch-and-solve has the same form---with $d_{\mathrm{eff}}(\epsilon \mu)$ in place of $d_{\mathrm{eff}}(\mu)$.  This is a massive difference.  Nystr{\"o}m PCG can produce solutions whose residual norm is close to machine precision; this type of successful computation is impossible with Nystr{\"o}m sketch-and-solve.

\subsubsection{Comparison to other randomized preconditioning methods}
\label{sec:preconditioner-comparison}
In this subsection we give a more comprehensive discussion of how Nystr{\"o}m PCG compares to prior work on randomized preconditioning \cite{avron2010blendenpik,lacotte2020effective,meng2014lsrn,rokhlin2008fast} based on sketch-and-precondition and related ideas. All these prior methods were developed for least squares problems. We summarize the complexity of each method for regularized least-squares problems in \cref{t-ComplexityComparison}. 

\begin{table}[h]
    \begin{center}
\footnotesize
    \caption{\label{t-ComplexityComparison} \textbf{Regularized least-squares: complexity of prior randomized preconditoning methods vs. Nystr{\"o}m PCG}. The table compares the complexity of Nystr{\"o}m PCG and state-of-the-art randomized preconditioning methods in the overdetermined case $n\geq d$, 
    assuming we can access $A$ only via matrix-vector products. 
    The sketch-and-precondition preconditioner is constructed from a sketch $SA$,  where $S\in \mathbb{R}^{m\times n}$ is a $(1\pm \gamma)$ Gaussian subspace embedding with sketch size $\Omega(d/\gamma)$ and $\gamma \in (0,1).$ 
    The time to compute the sketch is $O(T_{\textrm{mv}}d / \gamma)$ and the iteration complexity follows from the argument in ~\cite[Sec ~2.6]{woodruff2014sketching}. 
    For AdaIHS we use a sketch constructed from a Gaussian subspace embedding with sketch size $O(d_{\textrm{eff}}(\mu)/\rho)$ where $\rho \in (0,0.18)$. The complexity of AdaIHS follows from ~\cite[Theorem 5]{lacotte2020effective}. 
    The complexity of Nystr{\"o}m PCG is derived from \cref{NysPCGEffDimThm,ConvergenceCorollary}.}
    \begin{tabular}{|c|c|c|}
    \hline
    \textbf{Method} & \thead{Complexity} & \thead{References}\\
    \hline
    Sketch-and-precondition & $O\left(T_\textrm{mv}d/\gamma + d^{3}/\gamma+T_{\textrm{mv}}\frac{\log\left(2/\epsilon\right)}{\log\left(1/\gamma\right)}\right)$ & \cite{avron2010blendenpik,meng2014lsrn,rokhlin2008fast} \\
    \hline
    AdaIHS & $\makecell{O\left((T_{\textrm{mv}}d_{\textrm{eff}}/\rho+dd_{\textrm{eff}}^{2}/\rho^{2})\log(d_{\textrm{eff}}/\rho)+T_{\textrm{mv}}\frac{\log(1/\epsilon)}{\log(1/{\rho})}\right)}$ & \cite{lacotte2020effective}\\
    \hline
    Nystr{\"o}m PCG & $O\left(T_\textrm{mv}d_{\textrm{eff}}+dd_{\textrm{eff}}^{2}+T_{\textrm{mv}}\log(2/\epsilon)\right)$ & This work  \\
    \hline
    \end{tabular}
\end{center}
\end{table}
The time to construct the sketch-and-precondition preconditioner is always larger than the Nystr{\"o}m preconditioner,
since $d_\text{eff} < d$ and $\gamma < 1$.
Indeed, constructing the preconditioner for sketch-and-precondition costs $\Omega(d^{3})$, 
which is the same as a direct method when $d = \Omega(n)$
and is prohibitive for high-dimensional problems. 
Hence Nystr{\"o}m PCG is amenable to problems with large $d$ and runs much faster than sketch-and-precondition whenever $d_{\textrm{eff}}(\mu)\ll d$.
The Nystr{\"o}m preconditioner also enjoys wider applicability then sketch-precondition: it applies to square-ish systems, whereas the others only work for strongly overdetermined or underdetermined problems.
In addition to sketch-and-precondition, Nystr{\"o}m PCG also enjoys better complexity than AdaIHS. AdaIHS has linear dependence in $d$, but possesses additional logarithmic factors and scales in terms of $d_{\textrm{eff}}(\mu)/\rho$ where $\rho<0.18$, leading to a worse runtime.  


In the context of kernel ridge regression (KRR), 
the random features method of \cite{avron2017faster}
may be viewed as a randomized preconditioning technique.
\cite{avron2017faster} prove convergence guarantees for the polynomial kernel with a potentially prohibitive sketch size $\ell = O(d_{\textrm{eff}}(\mu)^{2})$. 
In contrast, Nystr{\"o}m PCG can be used for KRR with any kernel and requires smaller sketch size $\ell = O(d_{\textrm{eff}}(\mu))$ to obtain fast convergence. 

In summary, Nystr{\"o}m PCG applies to a wider class of problems than prior randomized preconditioners and enjoys stronger theoretical guarantees for regularized problems. 
Nystr{\"o}m PCG also outperforms other randomized preconditioners numerically (\cref{section:NumericalExperiments}). 

\subsubsection{Block Nystr{\"o}m PCG}

We can also use the Nystr{\"o}m preconditioner with the block CG algorithm~\cite{o1980block} to solve regularized linear systems with multiple right-hand sides.  For this approach, we also use an orthogonalization pre-processing proposed in~\cite{feng1995block} that ensures numerical stability without further orthogonalization steps during the iteration.

\subsection{Analysis of Nystr{\"o}m PCG}

We now turn to the analysis of the randomized Nystr{\"o}m preconditioner $P$.
\cref{NysPCGEffDimThm} provides a bound for the rank $\ell$ of the Nystr{\"o}m preconditioner that reduces the condition number of $A_{\mu}$ to a constant.  In this case, we deduce that Nystr{\"o}m PCG converges rapidly (\cref{ConvergenceCorollary}).

\begin{theorem}[Nystr{\"o}m preconditioning]
\label{NysPCGEffDimThm}
Suppose we construct the Nystr{\"o}m preconditioner $P$ in \cref{eq-NysPre} using \cref{alg:RandNysAppx} with sketch size $\ell = 2\,\lceil 1.5\,d_{\textup{eff}}(\mu)\rceil+1$.
Using $P$ to precondition the regularized matrix $A_{\mu}$ results in the condition number bound
\[
\E\big[\kappa_{2}(P^{-1/2}A_{\mu}P^{-1/2}) \big] < 28.
\]
\end{theorem}

\noindent
The proof of \cref{NysPCGEffDimThm} may be found in \cref{Section:NysPCGEffDimThmProof}.

\cref{NysPCGEffDimThm} has several appealing features. 
Many other authors have noticed that the effective dimension controls 
sample size requirements for particular applications 
such as discriminant analysis \cite{chowdhury2018randomized},
ridge regression \cite{lacotte2020effective}, and kernel ridge regression \cite{el2014fast,bach2013sharp}. 
In contrast, our result holds for any regularized linear system.

Our argument makes the role of the effective dimension conceptually simpler,
and it leads to explicit, practical parameter recommendations.
Indeed, the effective dimension $d_{\mathrm{eff}}(\mu)$ is essentially
the same as the sketch size $\ell$ that makes the approximation error
$\|A - \hat{A}_{\mathrm{nys}}\|$ proportional to $\mu$. 
In previous arguments, such as those in \cite{el2014fast,bach2013sharp,chowdhury2018randomized},
the effective dimension arises because the authors reduce the analysis to approximate matrix multiplication~\cite{cohen2015optimal}, which produces inscrutable constant factors.

\Cref{NysPCGEffDimThm} ensures that Nystr{\"o}m PCG converges quickly.
\begin{corollary}[Nystr{\"o}m PCG: Convergence]
\label{ConvergenceCorollary}
Define $P$ as in \cref{NysPCGEffDimThm}, and condition on the event 
$\big\{\kappa_2\left(P^{-1/2}A_\mu P^{-1/2}\right)\leq 56\big\}$.
Let $M = P^{-1/2}A_{\mu}P^{-1/2}$.  If we initialize \cref{alg:NysPCG}
with initial iterate $x_0 = 0$, then the relative error $\delta_t$ in the iterate $x_t$ satisfies
\[
\delta_t = \frac{\|x_{t}-x_{\star}\|_M}{\|x_{\star}\|_{M}}
    < 2\cdot \left(0.77 \right)^{t}
    \quad\text{where $A_{\mu} x_{\star} = b$.}
\] 
In particular, after $t = \lceil{3.8 \log(2/\epsilon)\rceil}$ iterations,
we have relative error $\delta_t<\epsilon$.
\end{corollary}

The proof of \cref{ConvergenceCorollary} is an immediate consequence of the standard convergence result for CG \cite[Theorem 38.5, p.~299]{trefethen1997numerical}. See~\cref{section:ConvergenceCorollaryPf}. 

\subsubsection{Analyzing the condition number}

The first step in the proof of \cref{NysPCGEffDimThm} is a deterministic
bound on how the preconditioner \eqref{eq-NysPre} reduces the condition number
of the regularized matrix $A_{\mu}$.  Let us emphasize that this bound is valid
for any rank-$\ell$ Nystr{\"o}m approximation, regardless of the choice of test matrix.  

\begin{proposition}[Nystr{\"o}m preconditioner: deterministic bound]
\label{NysPCGThm}
Let $\hat{A} = U \hat{\Lambda} U^T$ be any rank-$\ell$ Nystr{\"o}m approximation, with $\ell$th largest eigenvalue $\hat{\lambda}_{\ell}$,
and let $E = A-\hat{A}$ be the approximation error.  Construct the Nystr{\"o}m preconditioner $P$ as in \eqref{eq-NysPre}.
Then the condition number of the preconditioned matrix $P^{-1/2}A_{\mu}P^{-1/2}$ satisfies
\begin{equation}
    \label{eq-CondNumBnd}
\begin{aligned}
     \max\bigg\{\frac{\hat{\lambda}_{\ell}+\mu}{\lambda_{n}+\mu},1\bigg\}
     &\leq\kappa_{2}(P^{-1/2}A_{\mu}P^{-1/2}) \\
     &\leq \left(\hat{\lambda}_{\ell}+\mu+\|E\|\right)\min\bigg\{\frac{1}{\mu},~
     \frac{\hat{\lambda}_{\ell} + \lambda_n + 2 \mu}{(\hat{\lambda}_\ell + \mu)(\lambda_n + \mu)}\bigg\}.
\end{aligned}
\end{equation}
\end{proposition} 
For the proof of \cref{NysPCGThm} see \cref{section:NysPCGThmProof}.

To interpret
the result, recall the expression~\cref{eq-OptCondNumber} for the condition
number induced by the optimal preconditioner.  \cref{NysPCGThm} shows that
the Nystr{\"o}m preconditioner may reduce the condition number
almost as well as the optimal preconditioner.

In particular, when $\| E \| = O(\mu)$, 
the condition number of the preconditioned system is bounded
by a constant, independent of the spectrum of $A$.  In this case,
Nystr{\"o}m PCG is guaranteed to converge quickly.

\subsubsection{The effective dimension and sketch size selection}

How should we choose the sketch size $\ell$ to guarantee that $\|E\| = O( \mu )$? 
\cref{NysExpCorr} shows how the error in the rank-$\ell$ randomized Nystr{\"o}m approximation
depends on the spectrum of $A$ through the eigenvalues of $A$ and the tail stable rank.
In this section, we present a lemma which demonstrates that the effective dimension
$d_{\mathrm{eff}}(\mu)$ controls both quantities.  As a consequence of this bound,
we will be able to choose the sketch size $\ell$ proportional
to the effective dimension $d_{\mathrm{eff}}(\mu)$.

Recall from \cref{EffDimDef} that the effective dimension of the matrix $A$
is defined as
\begin{equation} \label{eqn:deff-proof}
d_{\textrm{eff}}(\mu) = \tr(A(A+\mu I)^{-1}) = \sum_{j=1}^{n}\frac{\lambda_j(A)}{\lambda_j(A)+\mu}.
\end{equation}
As previously mentioned, $d_{\textrm{eff}}(\mu)$ may be viewed as a smoothed count of the eigenvalues
larger than $\mu$. Thus, one may expect that $\lambda_{k}(A) \lesssim \mu$
for $k \gtrsim d_{\textrm{eff}}(\mu)$. This intuition is correct, and it forms
the content of \cref{KeyLemma}.

\begin{lemma}[Effective dimension]
\label{KeyLemma}
Let $A\in \mathbb{S}_{n}^+(\mathbb{R})$ with eigenvalues $\lambda_{1}\geq\lambda_{2}\geq\cdots\geq\lambda_{n}$. Let $\mu > 0$ be regularization parameter,
and define the effective dimension as in~\cref{eqn:deff-proof}.  The following statements hold.
\begin{enumerate}
    \item \label{l-eff-max}
    Fix $\gamma>0$. If $j\geq (1+\gamma^{-1})d_{\textup{eff}}(\mu)$, then $\lambda_{j}\leq \gamma\mu$.
    \item \label{l-eff-avg} If $k\geq d_{\textup{eff}}(\mu)$, then
    $k^{-1} \sum_{j>k}\lambda_j \leq (d_{\textup{eff}}(\mu)/k) \cdot \mu$. 
\end{enumerate}
\end{lemma}
The proof of \cref{KeyLemma} may be found in \cref{section:KeyLemmaProof}.

 \cref{KeyLemma}, \cref{l-eff-max} captures the intuitive fact that there are no more than $2 d_{\mathrm{eff}}(\mu)$ eigenvalues larger than $\mu$.  Similarly,
 \cref{l-eff-avg} states that the effective dimension controls the sum
 of all the eigenvalues whose index exceeds the effective dimension.
 It is instructive to think about the meaning of these results
 when $d_{\textrm{eff}}(\mu)$ is small.

\subsubsection{Proof of \cref{NysPCGEffDimThm}}
\label{Section:NysPCGEffDimThmProof}
 We are now prepared to prove \cref{NysPCGEffDimThm}.
 The key ingredients in the proof are \cref{NysExpBound}, \cref{NysPCGThm}, and \cref{KeyLemma}.
 
\begin{proof}[Proof of \cref{NysPCGEffDimThm}]
Fix the sketch size $\ell = 2 \, \lceil 1.5 \, d_{\mathrm{eff}}(\mu) \rceil + 1$.
Construct the rank-$\ell$ randomized Nystr{\"o}m approximation $\hat{A}_{\mathrm{nys}}$
with eigenvalues $\hat{\lambda}_j$.  Write $E = A - \hat{A}_{\mathrm{nys}}$ for the approximation error.
Form the preconditioner $P$ via~\cref{eq-NysPre}.  We must bound the expected condition number
of the preconditioned matrix $P^{-1/2} A_{\mu} P^{-1/2}$
 
First, we apply \cref{NysPCGThm} to obtain a deterministic bound that is valid for any rank-$\ell$
Nystr{\"o}m preconditioner:
\[
\kappa_{2}(P^{-1/2}A_{\mu}P^{-1/2})\leq \frac{\hat{\lambda}_{\ell}+\mu+\|E\|}{\mu}
    \leq 2 + \frac{\|E\|}{\mu}.
\]
The second inequality holds because $\hat{\lambda}_{\ell} \leq \lambda_{\ell} \leq \mu$.
This is a consequence of \cref{NysPropLemma}, \cref{it:nys-eigs}
and \cref{KeyLemma}, \cref{l-eff-max} with $\gamma = 1$.
We rely on the fact that $\ell \geq 2 \, d_{\mathrm{eff}}(\mu)$.

Decompose $\ell = 2p-1$ where $p = \lceil1.5\,d_{\textrm{eff}}(\mu)\rceil+1$.
Take the expectation, and invoke \cref{NysExpCorr} to obtain 
\[\E\big[\kappa_{2}(P^{-1/2}A_{\mu}P^{-1/2})\big]
    \leq 2 + \left(3+\frac{4\textrm{e}^{2}}{p}\textrm{sr}_{p}(A)\right) (\lambda_{p} / \mu).\]
By definition, $\textrm{sr}_{p}(A) \cdot \lambda_{p} = \sum_{j \geq p}\lambda_{j}$.
To complete the bound, apply \cref{KeyLemma} twice.  We use \cref{l-eff-max} with $\gamma = 2$ and \cref{l-eff-avg} with $k = p-1 = \lceil 1.5\,d_{\textrm{eff}}(\mu)\rceil$ to reach
\[
\E\big[\kappa_{2}(P^{-1/2}A_{\mu}P^{-1/2})\big]\leq 2+\frac{3\cdot 2\mu+4\textrm{e}^{2}\cdot 2\mu/3}{\mu}<2+26 = 28,
\]
which is the desired result.
\end{proof}

\subsection{Practical parameter selection}
\label{sec:adaptive-rank}

In practice, we may not know the regularization parameter $\mu$ in advance,
and we rarely know the effective dimension $d_{\mathrm{eff}}(\mu)$.  As a consequence,
we cannot enact the theoretical recommendation for the rank of the Nystr{\"o}m
preconditioner: $\ell = 2 \, \lceil 1.5\,d_{\mathrm{eff}}(\mu) \rceil + 1$.
Instead, we need an adaptive method for choosing the rank $\ell$.  Below,
we outline three strategies.

\subsubsection{Strategy 1: Adaptive rank selection by a posteriori error estimation}
\label{section:ErrorRankDoubling}

The first strategy uses the posterior condition number estimate adaptively
in a procedure the repeatedly doubles the sketch size $\ell$ as required.
 Recall that
 \cref{NysPCGThm} controls the condition number
 of the preconditioned system:
\begin{equation}
\label{eq-kappa-reg-bnds}
\kappa_{2}(P^{-1/2}A_{\mu}P^{-1/2}) \leq \frac{\hat{\lambda}_{\ell}+\mu+\|E\|}{\mu}
\quad\text{where $E = A - \hat{A}_{\mathrm{nys}}$.}
\end{equation} 
We get $\hat{\lambda}_{\ell}$ for free from \cref{alg:RandNysAppx} and we can compute the error $\|E\|$ inexpensively with the randomized power method \cite{kuczynski1992estimating};
see \cref{alg:RandPowErrEst} in \cref{section:AdaRankSelection}. 
Thus, we can ensure the condition number is small by making $\|E\|$ and $\hat{\lambda}_{\ell}$ fall below some desired tolerance.
The adaptive strategy proceeds to do this as follows.
We compute a randomized Nystr{\"o}m approximation with initial sketch size $\ell_{0}$, and we estimate the error $\| E \|$ using randomized powering. 
If $\|E\|$ is smaller than a prescribed tolerance, we accept the rank-$\ell_0$ approximation. If the tolerance is not met, then we double the sketch size, update the approximation, and estimate $\|E\|$ again. 
The process repeats until the estimate for $\|E\|$ falls below the tolerance or it breaches a threshold $\ell_{\textrm{max}}$ for the maximum sketch size.  \cref{alg:AdaRandNysAppx} usings the following stopping criterions $\| E \| \leq \tau\mu$ and $\hat{\lambda}_{\ell}\leq \tau\mu/11$ for a modest constant $\tau$. The stopping criterion on $\hat{\lambda}_{\ell}$ does not seem to be necessary in practice, as it is usually an order of magnitude small than $\|E\|$, but it is needed for \cref{AdaNysRankThm}.   
Based on numerical experience, we recommend a choice $\tau\in [1,100]$.  
For full algorithmic details of adaptive rank selection by estimating $\|E\|$,
 see \ref{alg:AdaRandNysAppx} in \ref{section:AdaRankSelection}. 
 
 The following theorem shows that with high probability, \cref{alg:AdaRandNysAppx} terminates with a modest sketch size in at most a logarithmic number of steps, and PCG with the resulting preconditioner converges rapidly.
 
 \begin{theorem}
 \label{AdaNysRankThm}
Run \cref{alg:AdaRandNysAppx} with initial sketch size $\ell_0$ and tolerance $\tau \mu$ where $\tau \geq 1 $, and let $\tilde{\ell} = 2\lceil 2d_{\textup{eff}}\left(\frac{\delta\tau\mu}{11}\right)\rceil+1$. Then with probability at least $1-\delta$:
\begin{itemize}
    \item[1.] \cref{alg:AdaRandNysAppx} doubles the sketch size at most
    $\lceil\log_{2}\left(\frac{\tilde{\ell}}{\ell_{0}}\right)\rceil$ times.
    \item[2.] The final sketch size $\ell$ satisfies
    \[\ell \leq 4\lceil2d_{\textup{eff}}\left(\frac{\delta\tau\mu}{11}\right)\rceil+2.\]
    \item[3.] With the preconditioner constructed from \cref{alg:AdaRandNysAppx}, Nystr{\"o}m PCG converges in at most $\lceil\frac{\log(2/\epsilon)}{\log(1/\tau_{0})}\rceil$ iterations, where $\tau_{0} = \frac{\sqrt{1+12\tau/11}-1}{\sqrt{1+12\tau/11}+1}$. 
\end{itemize}
\end{theorem}

\cref{AdaNysRankThm} immediately implies the following concrete guarantee.
\begin{corollary}
Set $\tau = 44$ and $\delta = 1/4$ in \cref{alg:AdaRandNysAppx} then with probability at least $3/4$:
\item[1.] \cref{alg:AdaRandNysAppx} doubles the sketch size at most
    $\lceil\log_{2}\left(\frac{\tilde{\ell}}{\ell_{0}}\right)\rceil$ times.
    \item[2.] The final sketch size $\ell$ satisfies
    \[\ell \leq 4\lceil2d_{\textup{eff}}(\mu)\rceil+2.\]
    \item[3.] With the preconditioner constructed from \cref{alg:AdaRandNysAppx}, Nystr{\"o}m PCG converges in at most $\lceil3.48\log(2/\epsilon)\rceil$ iterations. 
\end{corollary}

\subsubsection{Strategy 2: Adaptive rank selection by monitoring $\hat{\lambda}_{\ell}/\mu$}
\label{section:RatioRankDoubling}

The second strategy is almost identical to the first, except we monitor the ratio $\hat{\lambda}_{\ell}/\mu$ instead of $\|E\|/\mu$. 
Strategy 2 doubles the approximation rank until $\hat{\lambda}_{\ell}/\mu$ falls below some tolerance (say, 10) 
or the sample size reaches the threshold $\ell_{\max}$.
The approach is justified by the following proposition which shows that once the rank $\ell$ is sufficiently large, with high probability, the exact condition number differs from the empirical condition number $(\hat{\lambda}_\ell+\mu)/\mu$ by at most a constant.
\begin{proposition}
\label{AdaRankRatProp}
Let $\tau\geq 0$ denote the tolerance and $\delta>0$ a given failure probability. Suppose the rank of the randomized Nystr{\"o}m approximation satisfies $\ell \geq 2\lceil 2d_{\textup{eff}}\left(\tau\mu)\right)\rceil+1$. Then 
\begin{equation}
    \mathbb{P}\left\{\left(\kappa_{2}(P^{-1/2}A_{\mu}P^{-1/2})-\frac{\hat{\lambda}_{\ell}+\mu}{\mu}\right)_{+}\leq\frac{\tau}{\delta}\right\}\geq 1-\delta,
\end{equation}
where $X_{+} = \max\{X,0\}.$
\end{proposition}
This strategy has the benefit of saving a bit of computation and is preferable when the target precision is not too important, eg, in machine learning problems where training error only loosely predicts test error.

\subsubsection{Strategy 3: Choose $\ell$ as large as the user can afford}
\label{section:LargeLStrategy}

The third strategy is to choose the rank $\ell$ as large as the user can afford.
This approach is coarse, and it does not yield any guarantees on the
cost of the Nystr{\"o}m PCG method.

Nevertheless, once we have constructed a rank-$\ell$ Nystr{\"o}m approximation we can combine the posterior estimate of the condition number used in strategy 1 with the standard convergence theory of PCG to obtain an upper bound for the iteration count of Nystr{\"o}m PCG.
This gives us advance warning
about how long it may take to solve the regularized linear system.
As in strategy 1 we compute the error $\| E \|$ in the condition number bound inexpensively with the randomized power method.

\section{Applications and experiments}
\label{section:NumericalExperiments}
In this section, we study the performance of Nystr{\"o}m PCG on real world data from three different applications: ridge regression, kernel ridge regression, and approximate cross-validation. 
The experiments demonstrate the effectiveness of the preconditioner and our strategies for choosing the rank $\ell$ compared to other algorithms in the literature:
on large datasets, we find that our method outperforms competitors by a factor of 5--10 (\Cref{t-rfregress_results} and \Cref{t-largekrr-results}).

\subsection{Preliminaries}
\label{section:ExperimentalPrelims}
We implemented all experiments in MATLAB R2019a and MATLAB R2021a on a server with 128 Intel Xeon E7-4850 v4 2.10GHz CPU cores and 1056 GB. Except for the very large scale datasets ($n\geq 10^{5})$, every numerical experiment in this section was repeated twenty times;
tables report the mean over the twenty runs, and the standard deviation (in parentheses) when it is non-zero. We highlight the best-performing method in a table in bold.     

We select hyperparameters of competing methods by grid search to optimize performance.
This procedure tends to be very charitable to the competitors, and it may not be representative of their real-world performance.
Indeed, grid search is computationally expensive, and it cannot be used as part of a practical implementation.
A detailed overview of the experimental setup for each application may be found in the appropriate section of \cref{section:Experimentaldetails}, and additional numerical results in \cref{section:AddNumerics}.

\subsection{Ridge regression}
\label{section:RidgeRegression}
In this section, we solve the ridge regression problem \eqref{eq-RidgeRegression} described in \cref{IntroRidgeEx} on some standard machine learning data sets (\cref{t-ridge-datasets}) from OpenML \cite{vanschoren2014openml} and LIBSVM \cite{chang2011libsvm}. 
We compare Nystr{\"o}m PCG to standard CG and two state-of-the-art randomized preconditioning methods, 
the sketch-and-precondition method of Rokhlin and Tygert~(R\&T)~\cite{rokhlin2008fast} and the Adaptive Iterative Hessian Sketch (AdaIHS)~\cite{lacotte2020effective}.

\subsubsection{Experimental overview}

  \begin{table}[t]
  \caption{\label{t-ridge-datasets} \textbf{Ridge regression datasets.}}
\footnotesize
\begin{center}
    \begin{tabular}{|c|c|c|}
    \hline
    \textbf{Dataset} & $\mathbf{n}$ & $\mathbf{d}$ \\
    \hline
    CIFAR-10 & 50,000 & 3,072 \\
    \hline
    Guillermo & 20,000 & 4,297 \\
    \hline 
    smallNorb-rf & 24,300 & 10,000 \\
    \hline
    shuttle-rf  & 43,300 & 10,000 \\
    \hline
    Higgs-rf & 800,000 & 10,000 \\
    \hline
    YearMSD-rf & 463,715 & 15,000\\
    \hline
    \end{tabular}
\end{center}
\end{table}

We perform two sets of experiments: computing regularization paths on CIFAR-10 and Guillermo, 
and random features regression \cite{rahimi2007random,rahimi2008uniform} 
on shuttle, smallNORB, Higgs and YearMSD with specified values of $\mu$. The values of $\mu$ may be found in \cref{section:RidgeRegressionDetails}.
We use Euclidean norm $\|r\|_{2}$ of the residual as our stopping criteria, with convergence declared when $\|r\|_2\leq 10^{-10}$. 
For both sets of experiments, we use Nystr{\"o}m PCG with adaptive rank selection, \cref{alg:AdaRandNysAppx} in \cref{section:AdaRankSelection}. For complete details of the experimental methodology, see \cref{section:RidgeRegressionDetails}. 

The regularization path experiments solve \cref{eq-RidgeRegression} over a regularization path $\mu = 10^{j}$ where $j = 3,\cdots,-6$. 
We begin by solving the problem for the largest $\mu$ (initialized at zero), and solve for progressively smaller $\mu$ with warm starting. For each value of $\mu$, every method is allowed at most 500 iterations to reach the desired tolerance. 

\subsubsection{Computing the regularization path}
\label{sec:RegPath}
\cref{fig:EffDimPlots} shows the evolution of $d_{\textrm{eff}}(\mu)$ along the regularization path. CIFAR-10 and Guillermo are both small, so we compute the exact effective dimension as a point of reference. 
We see that we reach the sketch size cap of $\ell_{\textrm{max}} = 0.5d$ for CIFAR-10 and $\ell_{\textrm{max}} = 0.4d$ for Guillermo when $\mu$ is small enough. For CIFAR-10, 
Nystr{\"o}m PCG chooses a rank much smaller than the effective dimension for small values of $\mu$. Nevertheless, the method still performs well (\cref{fig:RidgeRegPathPlots}). 

\begin{figure}[t]
   \centering
     \begin{subfigure}[b]{0.48\textwidth}
         \centering
         \includegraphics[width=\textwidth]{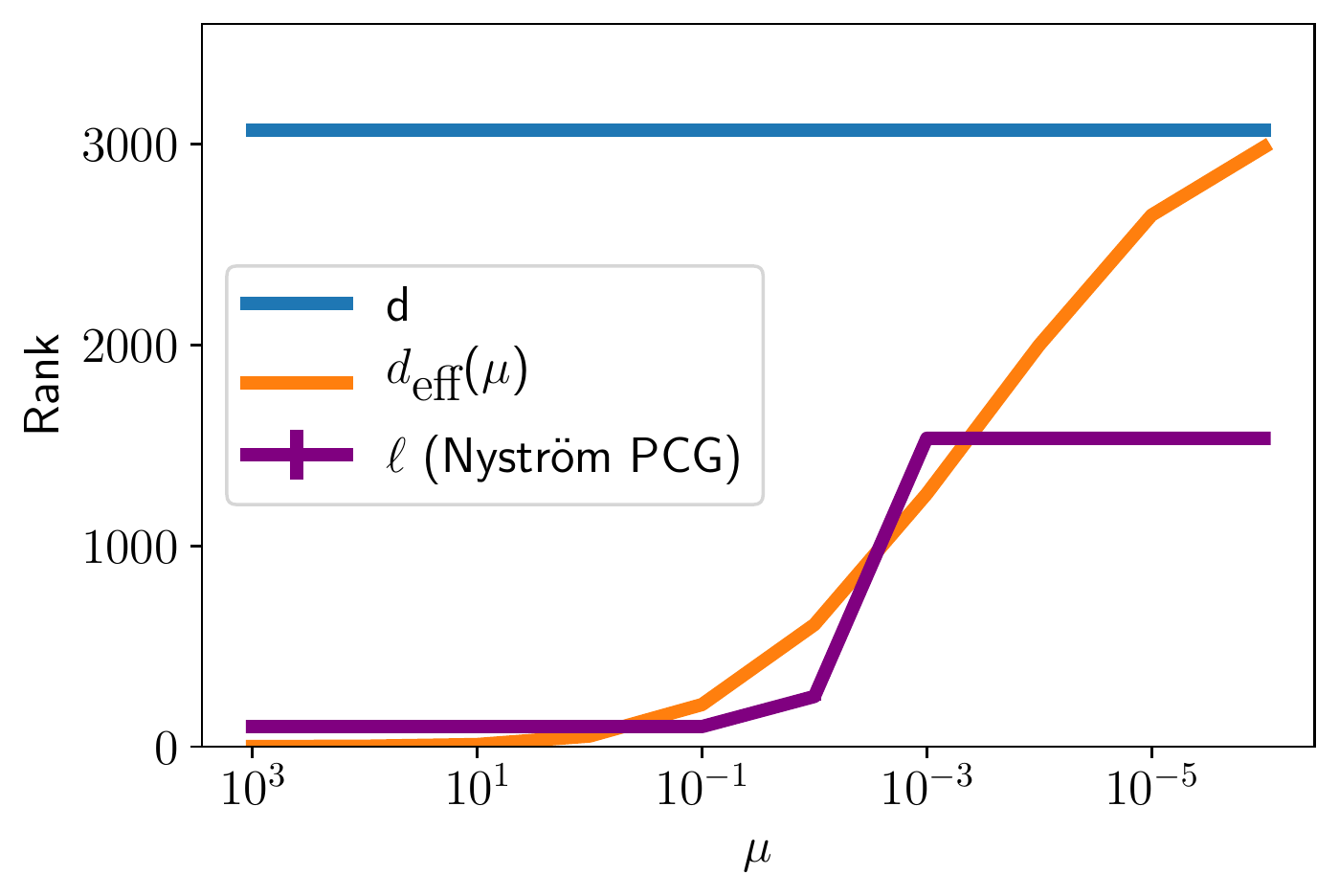}
         \caption{CIFAR 10}
         \label{fig:CifarEffDim}
     \end{subfigure}
     \begin{subfigure}[b]{0.48\textwidth}
         \centering
         \includegraphics[width=\textwidth]{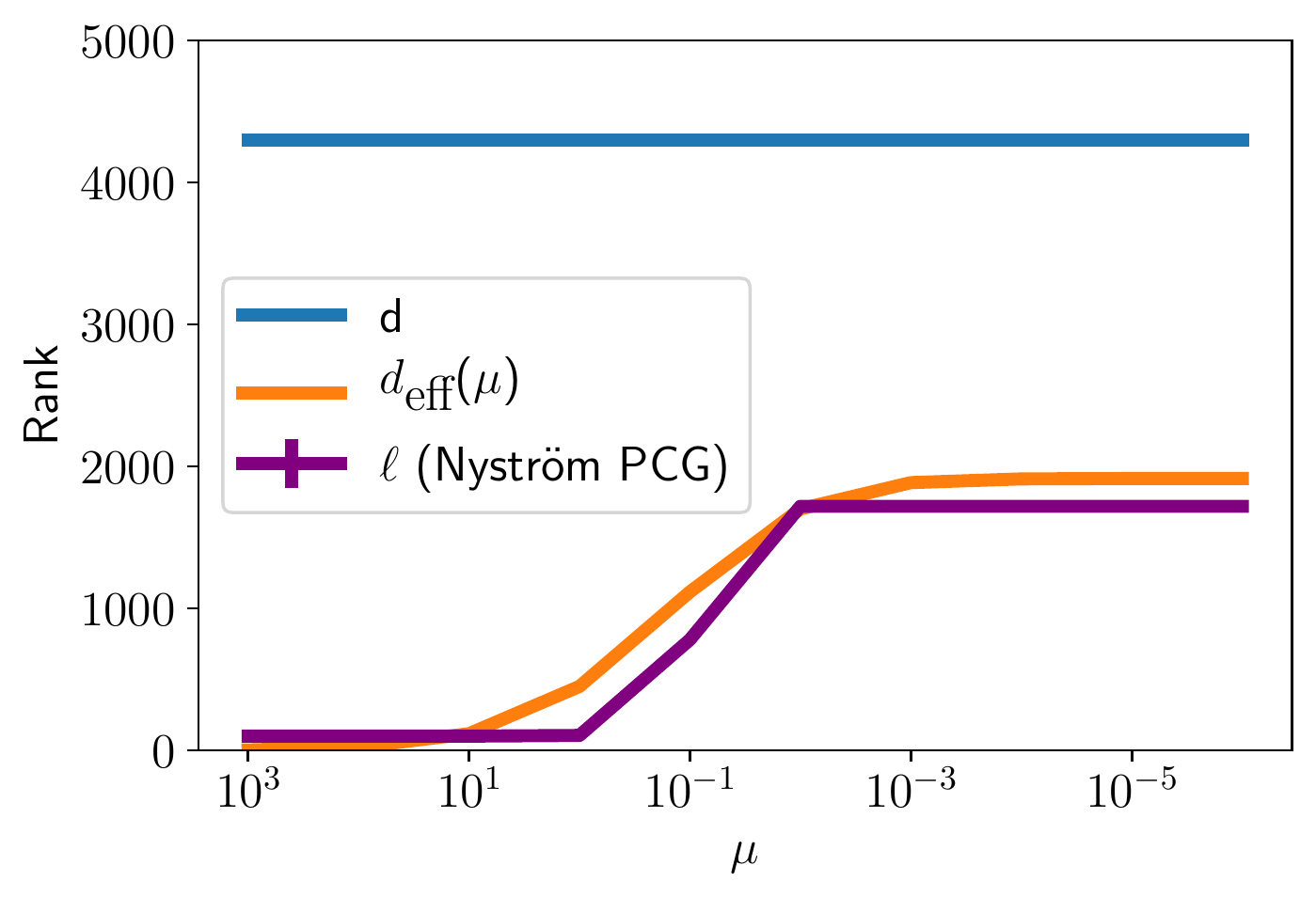}
         \caption{Guillermo}
         \label{fig:GuilEffDim}
     \end{subfigure}
\caption{\label{fig:EffDimPlots} \textbf{Ridge regression: Adaptive sketch size selection.}
Nystr{\"o}m PCG with adaptive rank selection (\cref{alg:AdaRandNysAppx}) selects a preconditioner whose rank is less than or equal to the effective dimension.We report error bars for the rank selected by the adaptive algorithm, however the variation is so small that the error bars aren't visible. Hence despite being inherently random, the adaptive algorithm's behavior is practically deterministic across runs. See \cref{sec:RegPath}.} 
\end{figure}

\begin{figure}[t]
   \centering
     \begin{subfigure}[b]{0.48\textwidth}
         \centering
         \includegraphics[width=\textwidth]{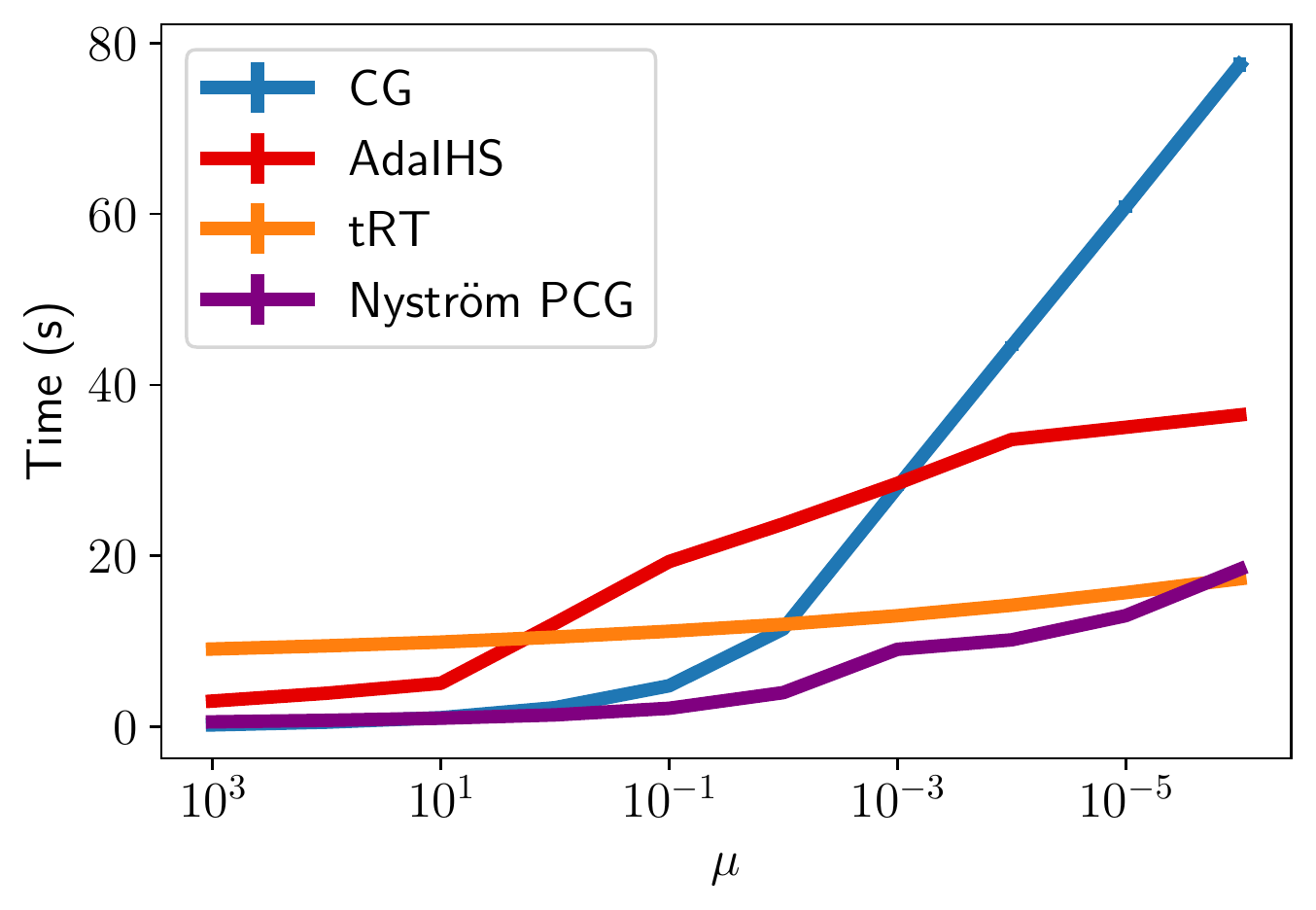}
         \caption{CIFAR 10: Runtime versus $\mu$}
         \label{fig:CifarCumTime}
     \end{subfigure}
     \begin{subfigure}[b]{0.48\textwidth}
         \centering
         \includegraphics[width=\textwidth]{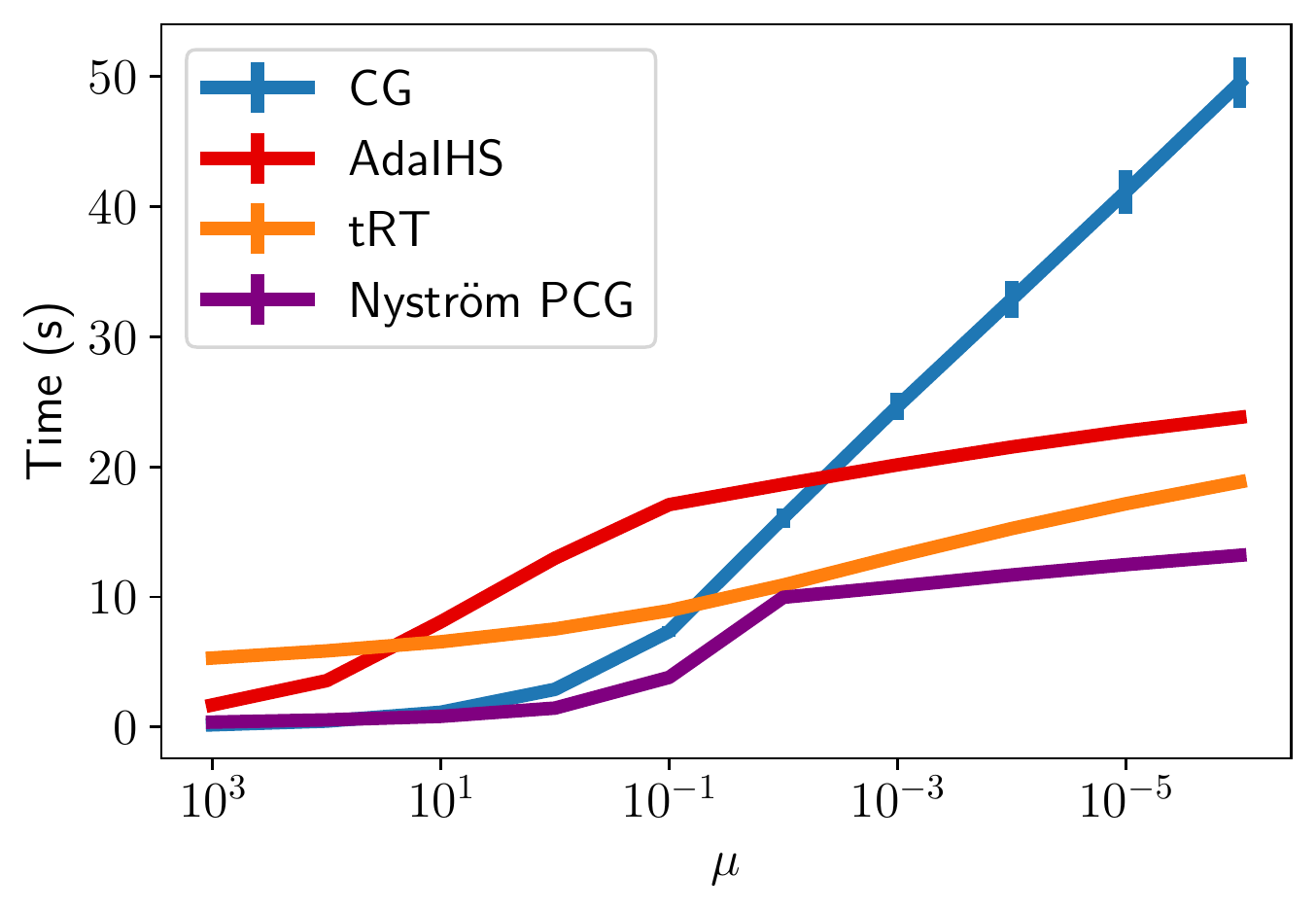}
         \caption{Guillermo: Runtime versus $\mu$}
         \label{fig:GuilCumTime}
     \end{subfigure}
     \centering
     \begin{subfigure}[b]{0.48\textwidth}
         \centering
         \includegraphics[width=\textwidth]{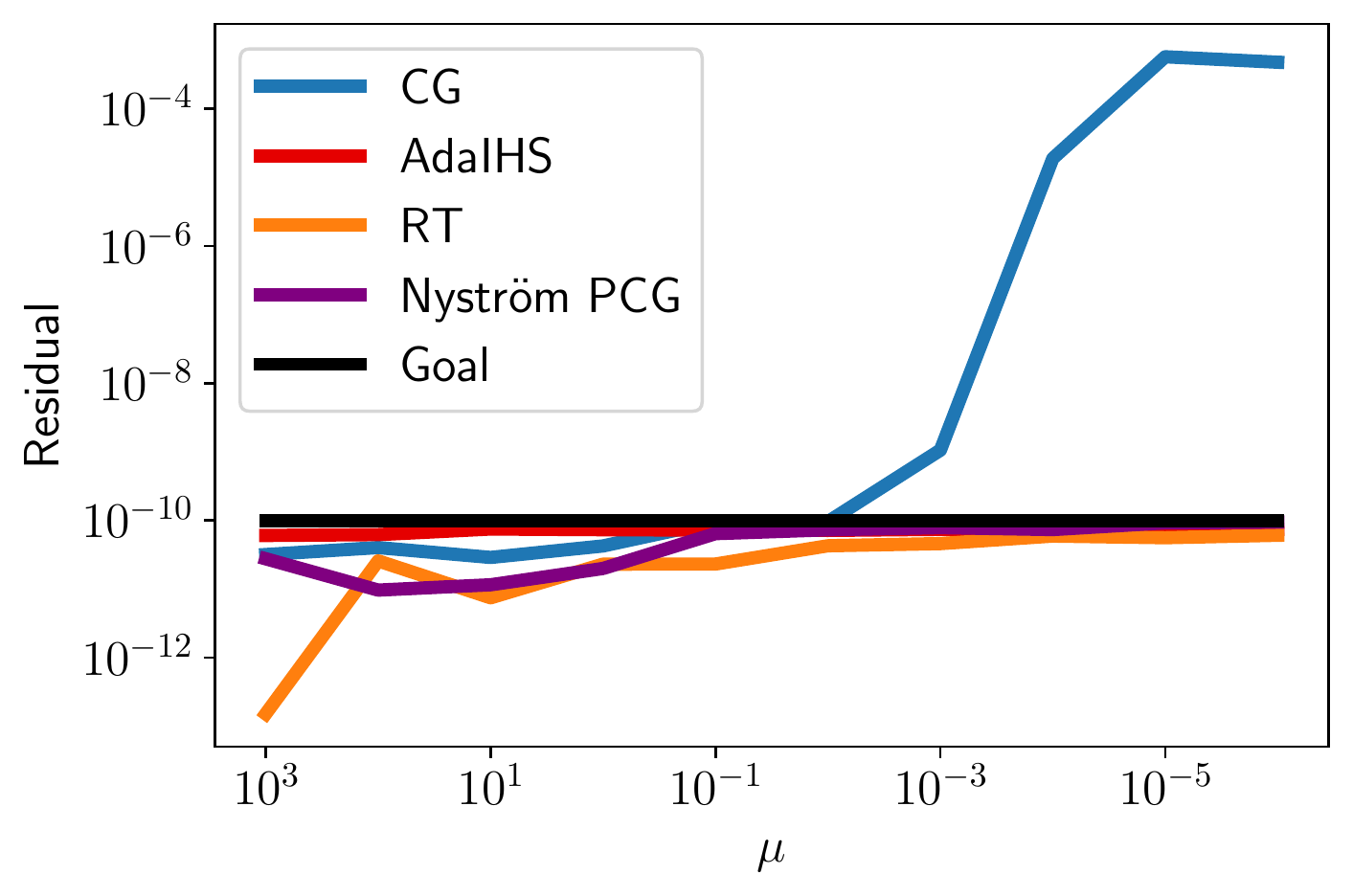}
         \caption{CIFAR 10: Residual versus $\mu$}
         \label{fig:CifarResid}
     \end{subfigure}
     \begin{subfigure}[b]{0.48\textwidth}
         \centering
         \includegraphics[width=\textwidth]{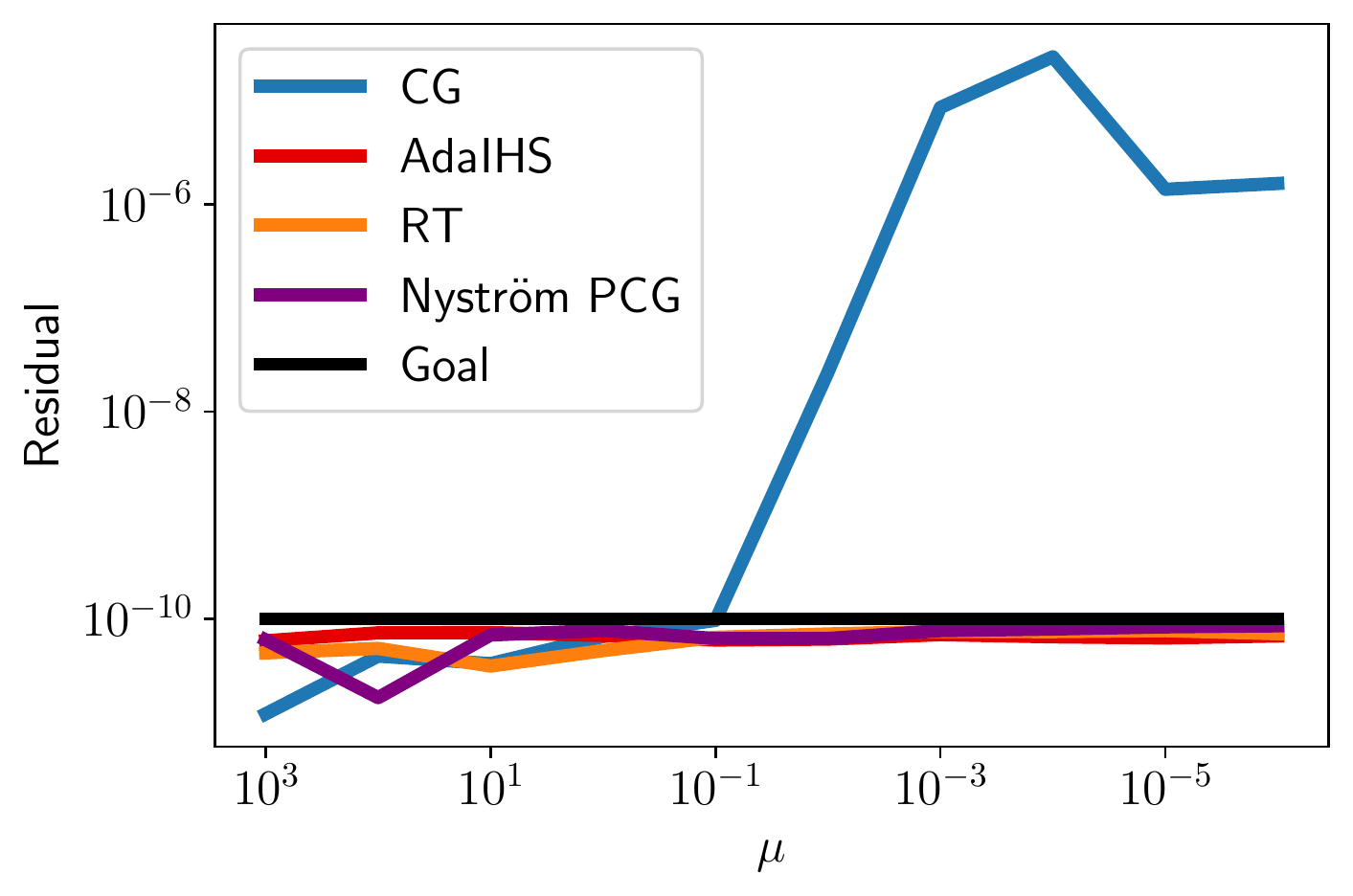}
         \caption{Guillermo: Residual versus $\mu$}
         \label{fig:GuilResid}
     \end{subfigure}     
\caption{\label{fig:RidgeRegPathPlots} \textbf{Ridge regression: Runtime and residual.}  Nystr{\"o}m PCG is either the fastest method, or it is competitive with the fastest method for all values of the regularization parameter $\mu$.  CG is generally the slowest method.  All the methods reliably achieve the target residual along the entire regularization path, except for ordinary CG at small values of $\mu$. See \cref{sec:RegPath}.} 
\end{figure}

 \cref{fig:RidgeRegPathPlots} show the effectiveness of each method for computing the entire regularization path. Nystr{\"o}m PCG is the fastest of the methods along most of the regularization path. 
 For CIFAR-10, Nystr{\"o}m PCG is faster than R\&T until the very end of the regularization path, when $d_{\textrm{eff}}(\mu) \approx d$. 
 That is, the $O(d^{3})$ cost of forming the R\&T preconditioner is not worthwhile unless $d_{\textrm{eff}}(\mu)\approx d$.
 We expect Nystr{\"o}m PCG to perform better on ridge regression problems with non-vanishing regularization.

AdaIHS is rather slow.  It increases the sketch size parameter several times along the regularization path.  Each time, AdaIHS must form a new sketch of the matrix, approximate the Hessian, and compute a Cholesky factorization. 

\subsubsection{Random features regression}

 \Cref{t-rfregress_results,t-rf-ErrCondEst} compare the performance of Nystr{\"o}m PCG, AdaIHS, and R\&T PCG for random features regression. \cref{t-rfregress_results} shows that Nystr{\"o}m PCG performs best on all datasets for all metrics.
The most striking feature is the difference between sketch sizes: 
AdaIHS and R\&T require much larger sketch sizes than Nystr{\"o}m PCG, leading to greater computation time and higher storage costs.
\cref{t-rf-ErrCondEst} contains estimates for $\|E\|$ and the condition number of the preconditioned system,
which explain the fast convergence.

\begin{table}[t]
\footnotesize
\caption{\label{t-rfregress_results} \textbf{Ridge regression: Nystr{\"o}m PCG versus AdaIHS and R\&T PCG.}  Nystr{\"o}m PCG outperforms AdaIHS and R\&T PCG in iteration and runtime complexity for both datasets.  Additionally, Nystr{\"o}m PCG requires much less storage.}
\begin{center}
\begin{tabular}{|c|c|c|c|c|}\hline
\thead{Dataset} & \thead{Method} & \thead{Final sketch size}& \thead{Number of \\ iterations} & \thead{Total \\
runtime (s)}\\ \hline
\multirow{3}{*}{shuttle-rf} & AdaIHS & $10,000$ & $66.9\hspace{2pt}(0.933)$ & $66.9\hspace{2pt}(5.27)$\\ \cline{2-5}
    & R\&T PCG & $20,000$ & $60.15$ & $242.6\hspace{2pt} (12.24) $ \\ \cline{2-5}
    & Nystr{\"o}m PCG & $800$ & $\mathbf{13.1\hspace{2pt}(1.47)}$ & $\mathbf{9.78\hspace{2pt}(0.943)}$ \\ \hline
\multirow{3}{*}{smallNORB-rf} & AdaIHS & $12,800$ & $38.7\hspace{2pt}(1.42)$ &  $41.0\hspace{2pt}(2.46)$\\ \cline{2-5}
    & R\&T PCG & $20,000$ & $34.5\hspace{2pt}(1.31)$ & $181.5\hspace{2pt}(6.53)$ \\ \cline{2-5}
    & Nystr{\"o}m PCG & $800$ & $\mathbf{31.5\hspace{2pt}(0.489)}$ & $\mathbf{6.67\hspace{2pt}(0.372)}$ \\ \hline
\multirow{3}{*}{YearMSD-rf} & AdaIHS & $30,000$ & $44$ & 1,327.3\\ \cline{2-5}
    & R\&T PCG & $30,000$ & $49$ & $766.5$ \\ \cline{2-5}
    & Nystr{\"o}m PCG & $2,000$ & $\mathbf{22}$ & $\mathbf{209.7}$ \\ \hline
\multirow{3}{*}{Higgs-rf} & AdaIHS & $6,400$ & $55$ & 1,052.7\\ \cline{2-5}
    & R\&T PCG & $20,000$ & $53$ & $607.4$ \\ \cline{2-5}
    & Nystr{\"o}m PCG & $800$ & $\mathbf{28}$ & $\mathbf{91.26}$ \\ \hline
\end{tabular}

\end{center}
\end{table}

\begin{table}[t]
    \begin{center}
\footnotesize
    \caption{\label{t-rf-ErrCondEst} \textbf{Ridge regression: Quality of Nystr{\"o}m preconditioner.}  We use Nystr{\"o}m PCG as outlined in \cref{section:LargeLStrategy}.  The first column gives an estimate for the error $\|E\|$ in the randomized Nystr{\"o}m approximation.  The second column gives an estimate for the condition number $\kappa_2(P^{-1/2}A_{\mu}P^{-1/2})$.}
    \begin{tabular}{|c|c|c|}
    \hline
    \textbf{Dataset} & \textbf{Estimate of $\| E\|$} & \thead{Estimated condition number \\ of preconditioned system}\\
    \hline
    shuttle-rf  & $\num{7.29e-13}$ & $4.17$ $(0.161)$ \\
    \hline
    smallNorb-rf & $\num{9.90e-3}$ & $18.5$ $(0.753)$ \\
    \hline
    YearMSD-rf & $\num{2.07e-4}$ & $22.7$\\
    \hline
    Higgs-rf & $\num{2.21e-3}$ & $23.8$\\
    \hline
    \end{tabular}
\end{center}
\end{table}

\subsection{Approximate cross-validation}
In this subsection we use our preconditioner to compute approximate leave-one-out cross-validation (ALOOCV), which requires solving a large linear system with multiple right-hand sides.

\label{section:ALOOCV}
\subsubsection{Background}
Cross-validation is an important machine-learning technique
to assess and select models and hyperparameters.
Generally, it requires re-fitting a model on many subsets of the data,
so can take quite a long time.
The worst culprit is leave-one-out cross-validation (LOOCV), which requires running an expensive training algorithm $n$ times. 
Recent work has developed approximate leave-one-out cross-validation (ALOOCV), a faster alternative
that replaces model retraining by a linear system solve \cite{giordano2019swiss,rad2020scalable,wilson2020approximate}. 
In particular, these techniques yield accurate and computationally tractable approximations to LOOCV.

To present the approach, we consider the infinitesimal jacknife (IJ) approximation to LOOCV \cite{giordano2019swiss,stephenson2020approximate}.
The IJ approximation computes
\begin{equation}
    \label{eq-IJ} 
    \tilde{\theta}_{\textrm{IJ}}^{n/j} = \hat{\theta}+\frac{1}{n}H^{-1}(\hat{\theta})\nabla_{\theta} l(\hat{\theta},a_{j}),
\end{equation}
where $H(\hat{\theta}) \in \R^{d \times d}$ is the Hessian of the loss function at the solution $\hat \theta$, for each datapoint $a_j$.
The main computational challenge is computing the inverse Hessian vector product $H^{-1}(\hat{\theta})\nabla_{\theta} l(\hat{\theta},a_{j})$.
When $n$ is very large, 
we can also subsample the data and average \cref{eq-IJ} over the subsample to estimate ALOOCV.
Since ALOOCV solves the same problem with several right-hand sides,
blocked PCG methods (here, Nystr{\"o}m blocked PCG) 
are the tool of choice
to efficiently solve for multiple right-hand sides at once. 
To demonstrate the idea, we perform numerical experiments on ALOOCV for logistic regression. 
The datasets we use are all from LIBSVM \cite{chang2011libsvm}; see \cref{t-logistic-datasets}.

\begin{table}[t]
\footnotesize
      \centering
        \caption{\label{t-logistic-datasets} \textbf{ALOOCV datasets and experimental parameters.}} 
        \begin{tabular}{|c|c|c|c||c|c|}
        \hline
        \textbf{Dataset} & $\mathbf{n}$ & $\mathbf{d}$ & $\mathbf{\%} \textbf{\textrm{nz}}(\mathbf{A})$
            & $\boldsymbol \mu$ & \thead{Initial sketch size}\\
        \hline
        Gisette & 6,000 & 5,000 & $99.1\%$ & 1, 1e-4 & 850\\
        \hline
        real-sim & 72,308 & 20,958 & $0.245\%$ & 1e-4, 1e-8 &  $500$ \\
        \hline 
        rcv1.binary & 20,242 & 47,236 & $0.157\%$ & 1e-4, 1e-8 & $500$\\
        \hline
        SVHN  & 73,257 & 3,072 & $100\%$  & 1, 1e-4 & 850\\
        \hline
        \end{tabular}
\end{table}
\subsubsection{Experimental overview}
We perform two sets of experiments in this section. The first set of experiments uses Gisette and SVHN to test the efficacy of Nystr{\"o}m sketch-and-solve. These datasets are small enough that we can factor $H(\theta)$ using a direct method. We also compare to block CG and block PCG with the computed Nystr{\"o}m approximation as a preconditioner. 
To assess the error due to an inexact solve for datapoint $a_j$, 
let $x_\star(a_j) = H^{-1}(\theta)\nabla_{\theta} l(\hat{\theta},a_{j})$. 
For any putative solution $\hat{x}(a_j)$,
we compute the relative error $\|\hat{x}(a_j)-x_{\star}(a_j) \|_{2}/\|x_{\star}(a_j)\|_{2}$.
We average the relative error over 100 data-points $a_j$.

 The second set of experiments uses the larger datasets real-sim and rcv1.binary and small values of $\mu$, the most challenging setting for ALOOCV. 
 We restrict our comparison to block Nystr{\"o}m PCG versus the block CG algorithm, as Nystr{\"o}m sketch-and-solve is so inaccurate in this regime.   We employ \cref{alg:AdaRandNysAppx} to construct the preconditioner for block Nystr{\"o}m PCG.
 
\subsubsection{Nystr{\"o}m sketch-and-solve}

As predicted, Nystr{\"o}m sketch-and-solve works poorly (\cref{t-gisette-svhn-results}).
When $\mu = 1$, the approximate solutions are modestly accurate,
and the accuracy degrades as $\mu$ decreases to $10^{-4}$. 
The experimental results agree with the theoretical analysis presented in \cref{Section:NystromSketchSolve}, which indicate that sketch-and-solve degrades as $\mu$ decreases.  In contrast, block CG and block Nystr{\"o}m PCG both provide high-quality solutions for each datapoint for both values of the regularization parameter. 
 
\begin{table}[t]
\footnotesize
    \caption{\label{t-gisette-svhn-results} \textbf{ALOOCV: Small datasets.}  The error for a given value of $\mu$ is the maximum relative error on $100$ randomly sampled datapoints, averaged over 20 trials.}
\begin{center}
    \begin{tabular}{|c|c|c|c|c|}
    \hline
    \thead{Dataset} &  $\bm{\mu}$ & \thead{Nystr{\"o}m \\ sketch-and-solve \\ } & \thead{Block CG } & \thead{Block \\ Nystr{\"o}m PCG } \\
    \hline
    Gisette  & $1$ & $\num{4.99e-2}$  & $\num{2.68e-11}$ & $\num{2.58e-12}$\\
    \hline
    Gisette  & $\num{1e-4}$ & $\num{1.22e-0}$ & $\num{1.19e-11}$ & $\num{6.59e-12}$ \\
    \hline
    SVHN & $1$& $\num{9.12e-5}$ & $\num{2.80e-13}$ & $\num{1.26e-13}$\\
    \hline 
    SVHN & $\num{1e-4}$ & $\num{3.42e-1}$ & $\num{2.01e-10}$ & $\num{1.41e-11}$\\
    \hline
    \end{tabular}

\end{center}
\end{table}

\subsection{Large scale ALOOCV experiments}
 \cref{t-rcv1-realsim-results-part-i} summarizes results for block Nystr{\"o}m PCG and block CG on the larger datasets. 
When $\mu = 10^{-4}$, block Nystr{\"o}m PCG offers little or no benefit over block CG
because the data matrices are very sparse (see \cref{t-logistic-datasets}) and the rcv1 problem is well-conditioned (see \cref{t-rcv1-realsim-results-part-ii}).

For $\mu = 10^{-8}$, block Nystr{\"o}m PCG reduces the number of iterations substantially,
but the speedup is negligible. 
The data matrix $A$ is sparse, which reduces the benefit of the Nystr{\"o}m method.
Block CG also benefits from the presence of multiple right-hand sides just as block Nystr{\"o}m PCG.
Indeed, O'Leary proved that the convergence of block CG depends on the ratio $(\lambda_{s}+\mu)/(\lambda_n+\mu)$, where $s$ is is the number of right-hand sides \cite{o1980block}. Consequently, multiple right-hand sides precondition block CG and accelerate convergence. 
We expect bigger gains over block CG when $A$ is dense.  

\begin{table}[t]
\footnotesize
\begin{center}
\caption{\label{t-rcv1-realsim-results-part-i} \textbf{ALOOCV: Large datasets}.  Block Nystr{\"o}m PCG outperforms block CG as $\mu$ becomes small.}
\begin{tabular}{|c|c|c|c|c|}\hline
\thead{Dataset} & $\bm{\mu}$ & \thead{Method} & \thead{Number of \\ iterations} & \thead{ 
Runtime (s)}\\ \hline
\multirow{2}{*}{rcv1} & $\num{1e-4}$ & Block CG & $12$  & $\mathbf{11.06\hspace{2pt}(0.874)}$\\ \cline{2-5}
    & $\num{1e-4}$ & Block Nystr{\"o}m PCG &  $\mathbf{10}$ & $11.87\hspace{2pt}(0.767)$ \\ \hline
\multirow{2}{*}{rcv1} & $\num{1e-8}$ & Block CG & $52$ &  $39.03\hspace{2pt}(2.97)$\\ \cline{2-5}
    & $\num{1e-8}$  & Block Nystr{\"o}m PCG & $\mathbf{15}$ & $\mathbf{24.1\hspace{2pt}(1.79)}$ \\ \hline
\multirow{2}{*}{realsim} & $\num{1e-4}$ & Block CG & $12$ &  $23.04\hspace{2pt}(2.04)$\\ \cline{2-5}
    & $\num{1e-4}$  & Block Nystr{\"o}m PCG & $\mathbf{8}$ & $\mathbf{19.05\hspace{2pt}(1.10)}$ \\ \hline
\multirow{2}{*}{realsim} & $\num{1e-8}$ & Block CG & $90$ & $163.7\hspace{2pt}(12.3)$\\ \cline{2-5}
    & $\num{1e-8}$  & Block Nystr{\"o}m PCG & $\mathbf{32}$ & $\mathbf{68.9\hspace{2pt}(5.30)}$ \\ \hline     
\end{tabular}

\end{center}
\end{table}

\subsection{Kernel ridge regression}
Our last application is kernel ridge regression (KRR),
a supervised learning technique that uses a kernel to 
model nonlinearity in the data. KRR leads to large dense linear systems that are challenging to solve.   
\label{section:KernelRidgeRegression}

\subsubsection{Background}
We briefly review KRR \cite{scholkopf2002learning}. Given a dataset of inputs $x_i \in \mathcal{D}$ and their corresponding outputs $b_i \in \R$ for $i = 1,\dots, n$
and a kernel function $\mathcal{K}(x,y)$,
KRR finds a function $f_{\star}: \mathcal{D} \to \R$ in the associated reproducing kernel Hilbert space $\mathcal{H}$ that best predicts the outputs for the given inputs. 
The solution $f_{\star}\in \mathcal{H}$ minimizes the square error 
subject to a complexity penalty:
\begin{equation}
   \label{KRR}
   f_\star = \argmin_{f\in \Hc} \quad \frac{1}{2n}\sum_{i=1}^{n}(f(x_{i})-b_{i})^{2}+\frac{\mu}{2}\|f\|^2_{\Hc},
\end{equation}
where $\|\cdot\|_\Hc$ denotes the norm on $\Hc$. 
Define the kernel matrix $K \in \R^{n \times n}$ with entries $K_{ij} = \mathcal{K}(x_i,x_j)$. 
The representer theorem \cite{steinwart2008support} states the solution to (\ref{KRR}) is 
\[
f_\star(x) = \sum_{i=1}^{n}\alpha_i \mathcal{K}(x,x_{i}),
\]
where $\alpha = (\alpha_1, \ldots, \alpha_n )$ solves the linear system
\begin{equation}
\label{KRRSys}
    (K+n\mu I)\alpha = b.
\end{equation}
Solving the linear system (\ref{KRRSys}) is the computational bottleneck of KRR. Direct factorization methods to solve \eqref{KRRSys} are prohibitive for large $n$ as their costs grow as $n^3$; for $n > 10^{4}$ or so, iterative methods are generally preferred.
However, $K$ is often extremely ill-conditioned, even with the regularization term $n\mu I$. As a result, CG for Problem \eqref{KRRSys} converges slowly.

\subsubsection{Experimental overview}
We use
Nystr{\"o}m PCG to solve several KRR problems derived from classification problems on real world datasets from \cite{chang2011libsvm,vanschoren2014openml}.
For all experiments, we use the Gaussian kernel $\mathcal{K}(x,y) = \exp(-\|x-y\|^{2} / (2\sigma^2))$. 
We compare our method to random features PCG, proposed in \cite{avron2017faster}. 
We do not compare to vanilla CG as it is much slower than Nystr{\"o}m PCG 
and random features PCG. 

All datasets either come with specified test sets, or we create one from a random $80$-$20$ split. 
The PCG tolerance, $\sigma$, and $\mu$ were all chosen to achieve good performance 
on the test sets (see \cref{t-krr-ranks} below). 
Both methods were allowed to run for a maximum of 500 iterations. 
The statistics for each dataset and the experimental parameters are given in 
\cref{t-KRR-datasets}. 

We run two sets of experiments. For the datasets with $n<10^{5}$ we run oracle random features PCG  against two versions of the Nystr{\"o}m PCG algorithm.
The first version uses the oracle best value of $\ell$ found by grid search (from the same grid used to select $m_{\textrm{rf}}$) to minimize the total runtime,
and the second is the adaptive algorithm described in \cref{section:RatioRankDoubling}. The grid for $\ell$ and $m_{\textrm{rf}}$ is restriced to less than $10,000$ to keep the preconditioners cheap to apply and store. 
The adaptive algorithm for each dataset was initialized at $\ell = 2,000$, which is smaller than $0.05n$ for all datasets.
For the datasets with $n\geq 10^{5}$, we restricted both $\ell$ and $m_{\textrm{rf}}$ to $1,000$, which corresponds to less than $0.01n$. We then run both algorithms till they reach the desired tolerance or the maximum number of iterations are reached.  

We use column sampling to construct the Nystr{\"o}m preconditioner for all KRR problems: on these problems, random projections take longer and yield similar performance (with somewhat lower variance).

\begin{table}[t]
\footnotesize
\begin{center}
    \caption{\label{t-KRR-datasets} \textbf{Kernel ridge regression datasets and experimental parameters.}}
    \begin{tabular}{|c|c|c|c|c|c|c|}
    \hline
    \textbf{Dataset} & $\mathbf{n}$ & $\mathbf{d}$ & $\mathbf{n}_{\textbf{\textrm{classes}}}$ & $\boldsymbol \mu$ & $\boldsymbol\sigma$ & \textbf{PCG tolerance}\\
    \hline
    ijcnn1  & 49,990 & 49 & 2 & 1e-6 & 0.5 & 1e-3 \\
    \hline
    MNIST & 60,000 & 784 & 10 & 1e-7 & 5 & 1e-4 \\
    \hline
    Sensorless & 48,509 & 48 & 11 & 1e-8 & 0.8 & 1e-4 \\
    \hline 
    SensIT & 78,823 & 100 & 3 & 1e-8 & 3 & 1e-3 \\
    \hline
    MiniBooNE & 104,052 & 50 & 2 & 1e-7& 5 & 1e-4 \\
    \hline
    EMNIST-Balanced & 105,280 & 784 & 47 &1e-6 & 8 & 1e-3\\
    \hline
    Santander & 160,000 & 200 & 2 & 1e-6 & 7 & 1e-3\\
    \hline
    \end{tabular}
\end{center}
\end{table}

\subsubsection{Experimental results}

\Cref{t-krr-results,t-largekrr-results,t-krr-ranks} summarize the results for the KRR experiments. 
\Cref{t-krr-results} shows that both versions of Nystr{\"o}m PCG deliver better performance than random features preconditioning on all the datasets considered. Nystr{\"o}m PCG also uses less storage. \Cref{t-largekrr-results} shows that Nystro{\"o}m PCG yields better performance than random features PCG on the larger scale datasets when both are restricted to ranks of $1,000$.  
\cref{t-krr-ranks} shows the adaptive strategy proposed in \cref{section:RatioRankDoubling} to select $\ell$ works very well. 
In contrast, it is difficult to choose $m_\textrm{rf}$ for random features preconditioning: the authors of \cite{avron2017faster} provide no guidance except for the polynomial kernel. 
  
\begin{table}[t]
\footnotesize
\begin{center}
\caption{\label{t-krr-results} \textbf{Kernel ridge regression: Iteration count and runtime.}  The adaptive and oracle Nystr{\"o}m PCG algorithms outperform oracle random features PCG in both time and iteration complexity.}
\begin{tabular}{|c|c|c|c|}\hline
\thead{Dataset} & \thead{Method} & \thead{Number of \\ iterations} & \thead{Total \\ runtime (s)}\\ \hline
\multirow{3}{*}{icjnn1} & Oracle random features PCG & $63.8(2.66)$ & $38.3(2.33)$\\ \cline{2-4}
    & Adaptive Nystr{\"o}m PCG & $43.7(1.77)$ & $\mathbf{32.0(1.47)}$ \\\cline{2-4}
    & Oracle Nystr{\"o}m PCG & $\mathbf{31.8(0.835)}$ & $33.3(1.60)$ \\ \hline
\multirow{3}{*}{MNIST} & Oracle random features PCG & $314.5(2.88)$ & $254.7(6.93)$\\ \cline{2-4}
    & Adaptive Nystr{\"o}m PCG & $78.5(17.65)$ & $148.1(46.39)$ \\\cline{2-4}
    & Oracle Nystr{\"o}m PCG & $\mathbf{77.9(2.08)}$ & $\mathbf{91.7(2.08)}$\\ \hline
\multirow{3}{*}{Sensorless} & Oracle random features PCG & $55.4(2.35)$ & $39.9(3.96)$ \\ \cline{2-4}
    & Adaptive Nystr{\"o}m PCG & $22.0(0.510)$ & $24.3(1.26)$\\\cline{2-4}
    & Oracle Nystr{\"o}m PCG& $\mathbf{21.7(0.571)}$ & $\mathbf{22.7(1.63)}$ \\ \hline
\multirow{3}{*}{SensIT} & Oracle random features PCG & $68.0(4.31)$ & $95.2(6.19)$ \\ \cline{2-4}
    & Adaptive Nystr{\"o}m PCG & $\mathbf{47.8(1.72)}$ & $70.1(2.43)$\\\cline{2-4}
    & Oracle Nystr{\"o}m PCG & $48.7(3.40)$ & $\mathbf{61.6(6.41)}$ \\ \hline
\end{tabular}
\end{center}
\end{table}

\begin{table}[t]
\footnotesize
\caption{\label{t-largekrr-results} \textbf{Kernel ridge regression regression: Nystr{\"o}m PCG fixed rank versus random features PCG fixed rank.}  Nystr{\"o}m PCG again outperforms random features PCG in iteration and time complexity.}
\begin{center}
\begin{tabular}{|c|c|c|c|}\hline
\thead{Dataset} & \thead{Method} & \thead{Number of \\ iterations} & \thead{Total \\
runtime (s)}\\ \hline
\multirow{2}{*}{MiniBooNE} & Random Features PCG & $92$ & $154.2$\\ \cline{2-4}
    & Nystr{\"o}m PCG & $\mathbf{72}$ & $\mathbf{137.4}$ \\ \hline
\multirow{2}{*}{EMNIST-Balanced} & Random Features PCG & $154$ &  $635.2$\\ \cline{2-4}
    & Nystr{\"o}m PCG & $\mathbf{32}$ & $\mathbf{268.4}$ \\ \hline 
\multirow{2}{*}{Santander} & Random Features PCG & $160$ &  $810.4$\\ \cline{2-4}
    & Nystr{\"o}m PCG & $\mathbf{31}$ & $\mathbf{164.8}$ \\ \hline      
\end{tabular}

\end{center}
\end{table}

\begin{table}[ht]
\footnotesize
\begin{center}
    \caption{\label{t-krr-ranks} \textbf{Kernel ridge regression: Oracle parameters and test errors.} The final ranks selected by the adaptive algorithm are almost in full agreement with the oracle ranks. Furthermore, the adaptive and oracle ranks for Nystr{\"o}m PCG are never larger than $m_{\textrm{rf}}$.}
    \begin{tabular}{|c|c|c|c|c|}
    \hline
    \textbf{Dataset} & \thead{Adaptive \\Nystr{\"o}m \\ final rank} & \thead{Oracle \\ Nystr{\"o}m \\ rank} & \thead{Oracle} $\mathbf{m_{\textrm{rf}}}$ & \thead{Test set error}  \\
    \hline
    ijcnn1  & $2,000$ & $3,000$ & $3,000$ & $1.25\%$\\
    \hline
    MNIST & $6,000\hspace{2pt} (1,716)$ & $4,000$ & $9,000$ & $1.22\%$ \\
    \hline
    Sensorless & $2,000$ & $2,000$ & $5,000$ & $2.01\%$ \\
    \hline 
    SensIT & $2,000$ & $2,000$ & $7,000$ & $12.83\%$ \\
    \hline
    MiniBooNE & NA & NA & NA & $7.93\%$ \\
    \hline
    EMNIST-Balanced &NA &NA &NA & $15\%$ \\
    \hline
    Santander & NA &NA &NA & $8.90\%$\\
    \hline
    \end{tabular}

\end{center}
\end{table}

\section{Conclusion}
We have shown that Nystr{\"o}m PCG delivers a strong benefit 
over standard CG both in the theory and in practice,
thanks to the ease of parameter selection,
on a range of interesting large-scale computational problems
including ridge regression, kernel ridge regression, and ALOOCV.
In our experience, Nystr{\"o}m PCG outperforms all generic methods
for solving large-scale dense linear systems with spectral decay.
It is our hope that this paper motivates further research on randomized preconditioning for solving large scale linear systems
and offers a useful speedup to practitioners.


\bibliographystyle{siamplain}
\bibliography{references}

\begin{thebibliography}{10}

\bibitem{el2014fast}
{\sc A.~Alaoui and M.~W. Mahoney}, {\em Fast randomized kernel ridge regression
  with statistical guarantees}, in NIPS, vol.~28, 2015, pp.~775--783.

\bibitem{avron2017faster}
{\sc H.~Avron, K.~L. Clarkson, and D.~P. Woodruff}, {\em Faster kernel ridge
  regression using sketching and preconditioning}, SIAM Journal on Matrix
  Analysis and Applications, 38 (2017), pp.~1116--1138.

\bibitem{avron2010blendenpik}
{\sc H.~Avron, P.~Maymounkov, and S.~Toledo}, {\em Blendenpik: Supercharging
  lapack's least-squares solver}, SIAM Journal on Scientific Computing, 32
  (2010), pp.~1217--1236.

\bibitem{bach2013sharp}
{\sc F.~Bach}, {\em Sharp analysis of low-rank kernel matrix approximations},
  in COLT, 2013, pp.~185--209.

\bibitem{bhatia2013matrix}
{\sc R.~Bhatia}, {\em Matrix analysis}, vol.~169, Springer Science \& Business
  Media, 2013.

\bibitem{chang2011libsvm}
{\sc C.-C. Chang and C.-J. Lin}, {\em {LIBSVM}: a library for support vector
  machines}, 2 (2011), pp.~1--27.

\bibitem{chowdhury2018randomized}
{\sc A.~Chowdhury, J.~Yang, and P.~Drineas}, {\em Randomized iterative
  algorithms for {Fisher} discriminant analysis}, in Uncertainty in Artificial
  Intelligence, PMLR, 2020, pp.~239--249.

\bibitem{cohen2015optimal}
{\sc M.~B. Cohen, J.~Nelson, and D.~P. Woodruff}, {\em Optimal approximate
  matrix product in terms of stable rank}, in International Colloquium on
  Automata, Languages, and Programming, 2016, pp.~11:1--11:14.

\bibitem{derezinski2020improved}
{\sc M.~Derezinski, R.~Khanna, and M.~W. Mahoney}, {\em Improved guarantees and
  a multiple-descent curve for column subset selection and the {Nystr{\"o}m}
  method}, in NeurIPS, vol.~33, 2020.

\bibitem{drineas2012fast}
{\sc P.~Drineas, M.~Magdon-Ismail, M.~W. Mahoney, and D.~P. Woodruff}, {\em
  Fast approximation of matrix coherence and statistical leverage}, The Journal
  of Machine Learning Research, 13 (2012), pp.~3475--3506.

\bibitem{feng1995block}
{\sc Y.~Feng, D.~Owen, and D.~Peri{\'c}}, {\em A block conjugate gradient
  method applied to linear systems with multiple right-hand sides}, Computer
  methods in applied mechanics and engineering, 127 (1995), pp.~203--215.

\bibitem{giordano2019swiss}
{\sc R.~Giordano, W.~T. Stephenson, R.~Liu, M.~I. Jordan, and T.~Broderick},
  {\em A swiss army infinitesimal jackknife}, in AISTATS, PMLR, 2019,
  pp.~1139--1147.

\bibitem{gittens2011spectral}
{\sc A.~Gittens}, {\em The spectral norm error of the naive {Nystr{\"o}m}
  extension}, arXiv preprint arXiv:1110.5305,  (2011).

\bibitem{gittens2013revisiting}
{\sc A.~Gittens and M.~W. Mahoney}, {\em Revisiting the {Nystr{\"o}m} method
  for improved large-scale machine learning}, The Journal of Machine Learning
  Research, 17 (2016), pp.~3977--4041.

\bibitem{golubvanloan2013}
{\sc G.~Golub and C.~Van~Loan}, {\em Matrix computations}, Johns Hopkins
  University Press, 2013.

\bibitem{gordon1985some}
{\sc Y.~Gordon}, {\em Some inequalities for {Gaussian} processes and
  applications}, Israel Journal of Mathematics, 50 (1985), pp.~265--289.

\bibitem{halko2011finding}
{\sc N.~Halko, P.-G. Martinsson, and J.~A. Tropp}, {\em Finding structure with
  randomness: Probabilistic algorithms for constructing approximate matrix
  decompositions}, SIAM {Review}, 53 (2011), pp.~217--288.

\bibitem{kuczynski1992estimating}
{\sc J.~Kuczy{\'n}ski and H.~Wo{\'z}niakowski}, {\em Estimating the largest
  eigenvalue by the power and {Lanczos} algorithms with a random start}, SIAM
  {Journal} on {Matrix Analysis} and {Applications}, 13 (1992), pp.~1094--1122.

\bibitem{lacotte2020effective}
{\sc J.~Lacotte and M.~Pilanci}, {\em Effective dimension adaptive sketching
  methods for faster regularized least-squares optimization}, in NeurIPS,
  vol.~33, 2020.

\bibitem{ledoux2013probability}
{\sc M.~Ledoux and M.~Talagrand}, {\em Probability in Banach Spaces:
  isoperimetry and processes}, Springer Science \& Business Media, 2013.

\bibitem{li2017algorithm}
{\sc H.~Li, G.~C. Linderman, A.~Szlam, K.~P. Stanton, Y.~Kluger, and
  M.~Tygert}, {\em Algorithm 971: An implementation of a randomized algorithm
  for principal component analysis}, ACM Transactions on Mathematical Software
  (TOMS), 43 (2017), pp.~1--14.

\bibitem{MartTroppSurvey}
{\sc P.-G. Martinsson and J.~A. Tropp}, {\em Randomized numerical linear
  algebra: Foundations and algorithms}, Acta Numerica, 29 (2020), pp.~403--572.

\bibitem{meng2014lsrn}
{\sc X.~Meng, M.~A. Saunders, and M.~W. Mahoney}, {\em Lsrn: A parallel
  iterative solver for strongly over-or underdetermined systems}, SIAM Journal
  on Scientific Computing, 36 (2014), pp.~C95--C118.

\bibitem{nakatsukasa2020fast}
{\sc Y.~Nakatsukasa}, {\em Fast and stable randomized low-rank matrix
  approximation}, arXiv preprint arXiv:2009.11392,  (2020).

\bibitem{o1980block}
{\sc D.~P. O'Leary}, {\em The block conjugate gradient algorithm and related
  methods}, Linear Algebra and its Applications,  (1980).

\bibitem{rad2020scalable}
{\sc K.~R. Rad, A.~Maleki, et~al.}, {\em A scalable estimate of the
  out-of-sample prediction error via approximate leave-one-out
  cross-validation}, Journal of the Royal Statistical Society: Series B
  (Statistical Methodology), 82 (2020), pp.~965--996.

\bibitem{rahimi2007random}
{\sc A.~Rahimi and B.~Recht}, {\em Random features for large-scale kernel
  machines}, in NIPS, vol.~20, 2007, pp.~1177--1184.

\bibitem{rahimi2008uniform}
{\sc A.~Rahimi and B.~Recht}, {\em Uniform approximation of functions with
  random bases}, in Allerton Conference on Communication, Control, and
  Computing, IEEE, 2008, pp.~555--561.

\bibitem{rokhlin2008fast}
{\sc V.~Rokhlin and M.~Tygert}, {\em A fast randomized algorithm for
  overdetermined linear least-squares regression}, Proceedings of the National
  Academy of Sciences, 105 (2008), pp.~13212--13217.

\bibitem{saad2003iterative}
{\sc Y.~Saad}, {\em Iterative methods for sparse linear systems}, SIAM, 2003.

\bibitem{sarlos2006improved}
{\sc T.~Sarlos}, {\em Improved approximation algorithms for large matrices via
  random projections}, in IEEE Symposium on Foundations of Computer Science
  (FOCS), IEEE, 2006, pp.~143--152.

\bibitem{scholkopf2002learning}
{\sc B.~Sch{\"o}lkopf and A.~J. Smola}, {\em Learning with kernels: support
  vector machines, regularization, optimization, and beyond}, MIT press, 2002.

\bibitem{steinwart2008support}
{\sc I.~Steinwart and A.~Christmann}, {\em Support vector machines}, Springer
  Science \& Business Media, 2008.

\bibitem{stephenson2020approximate}
{\sc W.~T. Stephenson and T.~Broderick}, {\em Approximate cross-validation in
  high dimensions with guarantees}, in AISTATS, PMLR, 2020, pp.~2424--2434.

\bibitem{stephenson2020LowRank}
{\sc W.~T. Stephenson, M.~Udell, and T.~Broderick}, {\em Approximate
  cross-validation with low-rank data in high dimensions}, in NeurIPS, vol.~33,
  2020.

\bibitem{trefethen1997numerical}
{\sc L.~N. Trefethen and D.~Bau~III}, {\em Numerical linear algebra}, vol.~50,
  SIAM, 1997.

\bibitem{tropp2017fixed}
{\sc J.~A. Tropp, A.~Yurtsever, M.~Udell, and V.~Cevher}, {\em Fixed-rank
  approximation of a positive-semidefinite matrix from streaming data}, in
  NIPS, vol.~30, 2017, pp.~1225--1234.

\bibitem{vanschoren2014openml}
{\sc J.~Vanschoren, J.~N. Van~Rijn, B.~Bischl, and L.~Torgo}, {\em {OpenML}:
  networked science in machine learning}, ACM SIGKDD Explorations Newsletter,
  15 (2014), pp.~49--60.

\bibitem{williams2001using}
{\sc C.~K. Williams and M.~Seeger}, {\em Using the {Nystr{\"o}m} method to
  speed up kernel machines}, in NIPS, vol.~13, 2001, pp.~682--688.

\bibitem{wilson2020approximate}
{\sc A.~Wilson, M.~Kasy, and L.~Mackey}, {\em Approximate cross-validation:
  Guarantees for model assessment and selection}, in AISTATS, PMLR, 2020,
  pp.~4530--4540.

\bibitem{woodruff2014sketching}
{\sc D.~P. Woodruff}, {\em Sketching as a tool for numerical linear algebra},
  Foundations and Trends{\textregistered} in Theoretical Computer Science, 10
  (2014), pp.~1--157.

\end{thebibliography}

\appendix
\section{Proofs of main results}

This appendix contains full proofs of the main results that are substantially novel (\cref{NysInvBnd,NysPCGThm,ConvergenceCorollary}).
The supplement contains proofs of results that are similar to existing work, but do not appear explicitly in the literature.

\subsection{Proof \cref{NysInvBnd}}
\label{section:NysSketchSolveProof}

This result contains the analysis of the Nystr{\"o}m sketch-and-solve method.
We begin with \cref{eq-NysInvExpErr}, which provides an error bound that compares the regularized inverse of
a psd matrix $A$ with the regularized inverse of the randomized Nystr{\"o}m approximation $\hat{A}_{\mathrm{nys}}$.
Since $0 \preceq \hat{A}_{\mathrm{nys}} \preceq A$, we can apply~\cref{GenInvErrBnd} to obtain a deterministic
bound for the discrepancy:
\[
\|(\hat{A}_{\textrm{nys}}+\mu I)^{-1}-(A+\mu I)^{-1}\|\leq \frac{1}{\mu}\frac{\|E\|}{\|E\|+\mu}
\quad\text{where $E = A - \hat{A}_{\mathrm{nys}}$.}
\]
The function $f(t) = t/(t+\mu)$ is concave, so we can take expectations and invoke Jensen's inequality to obtain
\[
\E\|(\hat{A}_{\textrm{nys}}+\mu I)^{-1}-(A+\mu I)^{-1}\|\leq \frac{1}{\mu}\frac{\E\|E\|}{\E\|E\|+\mu}.
\]
Inserting the bound \cref{eq-RandNysAppxErrBnd} on $\E \| E \|$ from \cref{NysExpCorr} gives
\[\E\|(\hat{A}_{\textrm{nys}}+\mu I)^{-1}-(A+\mu I)^{-1}\|
    \leq \frac{1}{\mu} \cdot \frac{(3+ 4\textup{e}^{2} \textup{sr}_{p}(A) / p)\lambda_{p}}
    {\mu + (3+4\textup{e}^{2}\textup{sr}_{p}(A)/p)\lambda_{p} }.\] 
To conclude, observe that the denominator of the second fraction exceeds
$\mu + \lambda_p$.


Now, let us establish \cref{eq-NysSSRelErr}, the error bound for Nystr{\"o}m sketch-and-solve.
Introduce the solution $\hat{x}$ to the Nystr{\"o}m sketch-and-solve problem and the solution
$x_{\star}$ to the regularized linear system:
\begin{align*}
    (\hat{A}_{\textrm{nys}}+\mu I)\hat{x} = b  \quad\text{and}\quad
    (A+\mu I)x_{\star} = b. 
\end{align*}
We may decompose the regularized matrix as $A+\mu I = \hat{A}_{\textrm{nys}}+\mu I+E$.
Subtract the two equations in the last display to obtain
\[(\hat{A}_{\textrm{nys}}+\mu I)(\hat{x}-x_{\star})-Ex_{\star} = 0.\]
Rearranging to isolate the error in the solution, we have
\[\hat{x}-x_{\star} = (\hat{A}_{\textrm{nys}}+\mu I)^{-1}Ex_{\star}.\]
Take the norm, apply the operator norm inequality, and use the elementary bound
$\| (\hat{A}_{\mathrm{nys}} + \mu I)^{-1} \| \leq \mu^{-1}$.  We obtain
\[\frac{\|\hat{x}-x_{\star}\|_{2}}{\|x_{\star}\|_{2}}\leq \frac{\|E\|}{\mu}.\]
Finally, take the expectation and repeat the argument used to control $\E\|E\|/\mu$
in the proof of \cref{NysPCGEffDimThm}.
\qed

\subsubsection{Proof of \cref{NysPCGThm}}
\label{section:NysPCGThmProof}

Let $\hat{A} = U \hat{\Lambda} U^T$ be an arbitrary rank-$\ell$ Nystr{\"o}m approximation
whose $\ell$th eigenvalue is $\hat{\lambda}_{\ell}$.
\Cref{NysPCGThm} provides a deterministic bound on the condition number of
the regularized matrix $A_{\mu}$ after preconditioning with
\[
P = \frac{1}{\hat{\lambda}_\ell + \mu} U (\hat{\Lambda} + \mu I) U^T + (I - UU^T). 
\]
We remind the reader that this argument is completely deterministic.

First, note that the preconditioned matrix $P^{-1/2}A_{\mu}P^{-1/2}$ is psd, so
\[\kappa_{2}(P^{-1/2}A_{\mu}P^{-1/2}) = \frac{\lambda_1(P^{-1/2}A_{\mu}P^{-1/2})}{\lambda_n(P^{-1/2}A_{\mu}P^{-1/2})}.\]
Let us begin with the upper bound on the condition number.  We have the decomposition
\begin{equation}
  \label{KeyDecomp}
  P^{-1/2}A_{\mu}P^{-1/2} = P^{-1/2}(\hat{A}+\mu I)P^{-1/2} +P^{-1/2}EP^{-1/2},
\end{equation}
owing to the relation $A_{\mu} = \hat{A} +\mu I+E$.
Recall that the error matrix $E$ is psd, so the matrix $P^{-1/2}EP^{-1/2}$ is also psd.

First, we bound the maximum eigenvalue.  Weyl's inequalities imply that
\[\lambda_1(P^{-1/2}A_{\mu}P^{-1/2})\leq \lambda_1(P^{-1/2}(\hat{A} +\mu I)P^{-1/2})+\lambda_{1}(P^{-1/2}EP^{-1/2}).\]
A short calculation shows that $\lambda_1(P^{-1/2}(\hat{A}+\mu I)P^{-1/2}) = \hat{\lambda}_{\ell}+\mu$.
When $\ell < n$, we have $\lambda_{1}(P^{-1}) = 1$.  Therefore,
\[
\lambda_{1}(P^{-1/2}EP^{-1/2}) = \lambda_{1}(P^{-1}E) \leq \lambda_{1}(P^{-1})\lambda_{1}(E) = \lambda_{1}(E) = \|E\|.
\]
In summary,
\begin{equation}
\label{TopEigValBnd}
    \lambda_1(P^{-1/2}A_{\mu}P^{-1/2})\leq \hat{\lambda}_{\ell}+\mu+\|E\|.
\end{equation}
For the minimum eigenvalue, we first assume that $\mu>0$.   Apply Weyl's inequality to \cref{KeyDecomp} to obtain
to obtain
\begin{equation}
\label{SmallEigValBnd1}
\begin{aligned}
\lambda_{n}(P^{-1/2}A_\mu P^{-1/2}) &\geq \lambda_{n}(P^{-1/2}(\hat{A}+\mu I)P^{-1/2})+\lambda_{n}(P^{-1/2}EP^{-1/2}) \\
    &\geq \lambda_{n}(P^{-1/2}(\hat{A}+\mu I)P^{-1/2}) = \mu.
\end{aligned}
\end{equation}
Combining \cref{TopEigValBnd} and \cref{SmallEigValBnd1}, we reach
\[
\kappa_{2}(P^{-1/2}A_{\mu}P^{-1/2}) \leq \frac{\hat{\lambda}_{\ell}+\mu+\|E\|}{\mu}.
\]
This gives a bound for the maximum in case $\mu > 0$.

If we only have $\mu \geq 0$, then a different argument is required for the smallest eigenvalue.
Assume that $A$ is positive definite, in which case $\hat{\lambda}_{\ell} > 0$.
By similarity,
\[\lambda_n(P^{-1/2}A_\mu P^{-1/2})
    = \lambda_n(A_{\mu}^{-1/2}PA_{\mu}^{-1/2}))
    = \frac{1}{\lambda_{1}(A_{\mu}^{-1/2}PA_{\mu}^{-1/2})}.
\]
It suffices to produce an upper bound for $\lambda_{1}(A_{\mu}^{-1/2}PA_{\mu}^{-1/2})$.
To that end, we expand
\begin{align*}
\lambda_{1}(A_{\mu}^{-1/2}PA_{\mu}^{-1/2})
    &= \lambda_{1}\bigg(A_{\mu}^{-1/2}\left(\frac{1}{\hat{\lambda}_{\ell}+\mu}(\hat{A}+\mu UU^{T})+(I-UU)^{T}\right)A_{\mu}^{-1/2}\bigg) \\
    &\leq \frac{1}{\hat{\lambda}_{\ell}+\mu}\lambda_{1}\left(A_{\mu}^{-1/2}(\hat{A}+\mu UU^{T})A_{\mu}^{-1/2}\right)\\
    &+\lambda_{1}\bigg(A_{\mu}^{-1/2}\left(I-UU^{T}\right)A_{\mu}^{-1/2}\bigg).
\end{align*}
The second inequality is Weyl's.
Since $\hat{A}\preceq A$, we have $\hat{A} +\mu UU^{T} \preceq A_{\mu}$.
The last display simplifies to
\[
\lambda_{1}(A_{\mu}^{-1/2}PA_{\mu}^{-1/2}) \leq \frac{1}{\hat{\lambda}_{\ell}+\mu}+\frac{1}{\lambda_{n}+\mu}.
\]
Putting the pieces together with \cref{TopEigValBnd}, we obtain
\[\kappa_{2}(P^{-1/2}A_{\mu}P^{-1/2}) \leq (\hat{\lambda}_{\ell}+\mu+\|E\|)\bigg(\frac{1}{\hat{\lambda}_{\ell}+\mu}+\frac{1}{\lambda_{n}+\mu}\bigg).\]
Thus,
\[\kappa_{2}(P^{-1/2}A_{\mu}P^{-1/2})\leq \left(\hat{\lambda}_{\ell}+\mu+\|E\|\right)\min\bigg\{\frac{1}{\mu},~
    \frac{\hat{\lambda}_{\ell} + \lambda_n + 2 \mu}{(\hat{\lambda}_\ell + \mu)(\lambda_n + \mu)}\bigg\}.
\] 
This formula is valid when $A$ is positive definite or when $\mu > 0$.

We now prove the lower bound on $\kappa_{2}(P^{-1/2}A_{\mu}P^{-1/2})$. Returning to \cref{KeyDecomp} and invoking Weyl's inequalities yields
\[ \lambda_1(P^{-1/2}A_{\mu}P^{-1/2})\geq \lambda_1(P^{-1/2}(\hat{A}+\mu I)P^{-1/2})+\lambda_{n}(P^{-1/2}EP^{-1/2})\geq \hat{\lambda}_{\ell}+\mu.\]
For the smallest eigenvalue we observe that
\[\lambda_{n}(P^{-1/2}A_{\mu}P^{-1/2}) = \lambda_{n}(A_{\mu}P^{-1}) \leq \lambda_{n}(A_{\mu})\lambda_{1}(P^{-1}) = \lambda_{n}+\mu.\]
Combining the last two displays, we obtain
\[
\frac{\hat{\lambda}_{\ell}+\mu}{\lambda_{n}+\mu}\leq \kappa_{2}(P^{-1/2}A_{\mu}P^{-1/2}).
\]
Condition numbers always exceed one, so
\[
\max\bigg\{\frac{\hat{\lambda}_{\ell}+\mu}{\lambda_{n}+\mu},1\bigg\}\leq \kappa_{2}(P^{-1/2}A_{\mu}P^{-1/2}).
\]
This point concludes the argument.

\subsubsection{Proof of \cref{KeyLemma}}
\label{section:KeyLemmaProof}

\Cref{KeyLemma} establishes the central facts about the effective dimension.
First, we prove \cref{l-eff-max}.
Fix a parameter $\gamma \geq 1$, and set $j_{\star} = \max\{1\leq j\leq n: \lambda_{j}>\gamma\mu\}$.
We can bound the effective dimension below by the following mechanism.
\[
d_{\textrm{eff}}(\mu)
    =  \sum^{n}_{j=1}\frac{\lambda_{j}}{\lambda_{j}+\mu}
    \geq \sum^{j_{\star}}_{j=1}\frac{\lambda_{j}}{\lambda_{j}+\mu}\geq  j_{\star} \cdot \frac{\lambda_{j_{\star}}}{\lambda_{j_{\star}}+\mu}.
\]
We have used the fact that $t \mapsto t/(1+t)$ is increasing for $t \geq 0$,
Solving for $j_{\star}$, we determine that
\[
j_{\star}\leq(1+\mu/\lambda_{j_{\star}})d_{\textrm{eff}}(\mu)<(1+\gamma^{-1})d_{\textrm{eff}}(\mu).
\]
The last inequality depends on the definition of $j_{\star}$.  This is the required result.

\cref{l-eff-avg} follows from a short calculation:
\[
\begin{aligned}
\frac{1}{k}\sum_{j>k}\lambda_{j}
    &= \frac{\lambda_k+\mu}{k}\sum_{j>k}\frac{\lambda_{j}}{\lambda_{k}+\mu}
    \leq \frac{\lambda_k+\mu}{k}\sum_{j>k}\frac{\lambda_{j}}{\lambda_{j}+\mu} \\
    &= \frac{\lambda_k+\mu}{k}\left(d_{\textrm{eff}}(\mu)-\sum^{k}_{j=1}\frac{\lambda_{j}}{\lambda_{j}+\mu}\right)
    \leq\frac{\lambda_k+\mu}{k}\left(d_{\textrm{eff}}(\mu)-\frac{k\lambda_{k}}{\lambda_{k}+\mu}\right) \\
    &= \frac{\mu \,d_{\textrm{eff}}(\mu)}{k} +\lambda_{k}\left(\frac{d_{\textrm{eff}}(\mu)}{k}-1\right)
    \leq \frac{\mu \, d_{\textrm{eff}}(\mu)}{k}.
\end{aligned}
\]
The last inequality depends on the assumption that $k\geq d_{\textrm{eff}}(\mu)$.

\subsection{Proof of \cref{ConvergenceCorollary}}
\label{section:ConvergenceCorollaryPf}

This result gives a bound for the relative error $\delta_t$ in the iterates of PCG.
Recall the standard convergence bound for CG~\cite[Theorem 38.5]{trefethen1997numerical}:
\[
\delta_{t}\leq 2 \, \bigg(\frac{\sqrt{\kappa_{2}(P^{-1/2}A_{\mu}P^{-1/2})}-1}{\sqrt{\kappa_{2}(P^{-1/2}A_{\mu}P^{-1/2})}+1}\bigg)^{t}.
\]
We conditioned on the event that $\{\kappa(P^{-1/2}A_{\mu}P^{-1/2}) \leq 56\}$.  On this event,
the relative error must satisfy
\[
\delta_{t} < 2 \, \left( \frac{\sqrt{56} - 1}{\sqrt{56} + 1}\right)^t
    \leq 2 \cdot (0.77)^{t}.
\]
Solving for $t$, we see that $\delta_t < \epsilon$ when $t\geq \lceil 3.8\log\left( 1/\epsilon \right)\rceil$.  This concludes the proof.
\subsection{Proof of \cref{AdaNysRankThm}}
\cref{AdaNysRankThm} establishes that with high probability \cref{alg:AdaRandNysAppx} terminates in a logarithmic number of steps, the sketch size remains $O(d_{\textrm{eff}}(\delta\tau\mu))$, and PCG with the preconditioner constructed from the output converges fast. 
\label{section:AdaNysRankThmPf}

\begin{proof}
We first recall that \cref{alg:AdaRandNysAppx} terminates when the event
\[\mathcal{E} = \{\|E\|\leq \tau\mu\}\cap\{\hat{\lambda}_{\ell}\leq \frac{\tau\mu}{11}\}\] 
holds. Observe that conditioned on $\mathcal{E}$, \cref{NysPCGThm} yields
\[\kappa_{2}(P^{-1/2}A_{\mu}P^{-1/2})\leq \frac{\hat{\lambda}_{\ell}+\mu+\|E\|}{\mu} \leq 1+(1+\frac{1}{11})\tau = 1+\frac{12}{11}\tau.\]
Statement 3 now follows from the above display and the standard convergence theorem for CG.

Now, if \cref{alg:AdaRandNysAppx} terminates with $N\leq \lceil\log_{2}\left(\tilde{\ell}/\ell_{0}\right)\rceil-1$ steps of sketch size doubling, then $\mathcal{E}$ holds with probability $1$.
Statement 3 then follows by our initial observation, while statements 1 and 2 hold trivially. Hence statements $1$-$3$ all hold if the algorithm terminates in $N\leq \lceil\log_{2}\left(\tilde{\ell}/\ell_{0}\right)\rceil-1$ steps.

Thus to conclude the proof, it suffices to show that if $N \geq \lceil\log_{2}\left(\tilde{\ell}/\ell_{0}\right)\rceil$, then $\mathcal{E}$ holds with probability at least $1-\delta$, which implies that statements $1$-$3$ hold with probability at least $1-\delta$, as above. 

We now show that $\mathcal{E}$  holds with probability at least $1-\delta$ when $N = \lceil\log_{2}\left(\tilde{\ell}/\ell_{0}\right)\rceil$. To see this note that when $N = \lceil\log_{2}\left(\tilde{\ell}/\ell_{0}\right)\rceil$, we have $\ell\geq \tilde{\ell}$. Consequently, we may
invoke \cref{NysExpBound} with $p = \lceil 2d_{\textup{eff}}\left(\frac{\delta\tau\mu}{11}\right)\rceil+1 $ and \cref{KeyLemma} to show
\begin{align*}
  \mathbb{E}[\|E\|] &\overset{(1)}{\leq} 3\lambda_{p}+\frac{2\textrm{e}^{2}}{p}\left(\sum^{n}_{j=p}\lambda_{j}\right)\\
  &\overset{(2)}{\leq} 3\frac{\delta \tau \mu}{11}+2\textrm{e}^{2}\frac{d_{\textrm{eff}}(\delta \tau \mu/11)}{p}\frac{\delta \tau \mu}{11}\\
  &\overset{(3)}{\leq} \frac{3}{11}\delta \tau \mu+\frac{2\textrm{e}^{2}}{2}\frac{\delta \tau \mu}{11}= \left(\frac{3+\textrm{e}^{2}}{11}\right)\delta \tau \mu \leq \delta \tau \mu.
\end{align*}
Where step (1) uses \cref{NysExpBound}, step (2) uses items 1 and 2 of \cref{KeyLemma} with $\gamma = 1$, and step (3) follows from $p\geq 2d_{\textrm{eff}}(\delta\tau\mu)$.
Thus,
\[\mathbb{E}[\|E\|] \leq \delta\tau \mu.\]
By Markov's inequality,
\[\mathbb{P}\{\|E\|>\tau \mu\}\leq \delta.\]
Hence $\{\|E\|\leq \tau \mu\}$ holds with probability at least $1-\delta$. Furthermore, by \cref{KeyLemma} we have $\{\hat{\lambda}_{\ell}\leq \delta \tau \mu/11\}$ with probability $1$ as $\hat{\lambda}_{\ell}\leq \lambda_{\ell}\leq \lambda_{p}$.  Thus when $N = \lceil\log_{2}\left(\tilde{\ell}/\ell_{0}\right)\rceil$, $\mathcal{E}$ holds with probability at least $1-\delta$, this immediately implies statements 1 and 3. Statement 2 follows as
\[\ell = 2^{N}\ell_0 \leq 2^{\log_{2}\left(\tilde{\ell}/\ell_{0}\right)+1}\ell_0 = 2\tilde{\ell} = 4\lceil2d_{\textup{eff}}\left(\frac{\delta\tau\mu}{11}\right)\rceil+2,\] 
where in the first inequality we used $\lceil x\rceil\leq x+1$, this completes the proof. 
\end{proof}

\subsection{Proof of \cref{AdaRankRatProp}}
\cref{AdaRankRatProp} shows once $\ell = \Omega\left(d_{\textrm{eff}}(\tau \mu)\right)$, then with high probability $\kappa_{2}(P^{-1/2}A_{\mu}P^{-1/2})$ differs from $(\hat{\lambda}_{\ell}+\mu)$ by at most a constant.
\begin{proof}
\cref{NysPCGThm} implies that
\[\left(\kappa_2(P^{-1/2}A_{\mu}P^{-1/2})-\frac{\hat{\lambda}_{\ell}+\mu}{\mu}\right)_{+}\leq \frac{\|E\|}{\mu}.\]
 Combining the previous display with Markov's inequality yields 
\[\mathbb{P}\left\{\left(\kappa_2(P^{-1/2}A_{\mu}P^{-1/2})-\frac{\hat{\lambda}_{\ell}+\mu}{\mu}\right)_{+}>\frac{\tau}{\delta}\right\} \leq \frac{\delta}{\tau}\frac{\mathbb{E}[\|E\|]}{\mu}.\]
Now, our choice of $\ell$ combined with \cref{NysExpBound} and \cref{KeyLemma} implies that $\mathbb{E}[\|E\|]\leq \tau\mu$. 
Hence we have
\[\mathbb{P}\left\{\left(\kappa_2(P^{-1/2}A_{\mu}P^{-1/2})-\frac{\hat{\lambda}_{\ell}+\mu}{\mu}\right)_{+}>\frac{\tau}{\delta}\right\} \leq \delta,\]
which implies the desired claim.
\end{proof}

\section{Proofs of additional propositions}
In this section we give the proofs for the propositions that are close to existing results but do not explicitly appear in the literature.
\subsubsection{Useful facts about Gaussian random matrices}
In this subsection we record some useful results about Gaussian random matrices that are necessary for the proof of \cref{NysExpBound}. The proof of \cref{NysExpBound} follows in \cref{Section:NysExpBoundPf}.
\begin{proposition}[\cite{halko2011finding,nakatsukasa2020fast}]
\label{GaussInvLemma}
Let $G$ be $(\ell-p)\times \ell$ standard Gaussian matrix with $\ell\geq 4$ and $2\leq p \leq \ell-2$. Then
\begin{equation}
    \left(\E\|G^{\dagger}\|^{2}_{F}\right)^{1/2} = \sqrt{\frac{\ell-p}{p-1}},
\end{equation}
and
\begin{equation}
 \left(\E\|G^{\dagger}\|^{2}\right)^{1/2} \leq \textup{e}\sqrt{\frac{\ell}{p^{2}-1}}.
\end{equation}
\end{proposition}
\begin{remark}
The first display in \cref{GaussInvLemma} appears in \cite{halko2011finding}, while the second display in \cref{GaussInvLemma} is due to \cite{nakatsukasa2020fast}.
\end{remark}
We also require one new result, which strengthens the improved version of Chevet's theorem due to Gordon \cite{gordon1985some}.
\begin{proposition}[Squared Chevet]
\label{SquaredChevet}
Fix matrices $S\in \mathbb{R}^{r\times m}$ and $T\in \mathbb{R}^{n\times s}$ and let $G\in\mathbb{R}^{m\times n}$ be a standard Gaussian matrix. Then
\[\E\|SGT\|^{2} \leq \left(\|S\| \|T\|_{F}+\|S\|_{F}\|T\|\right)^{2}.\]
\end{proposition}
\iftutorial
We defer the proof of \cref{SquaredChevet} to \cref{section:SquaredChevetPf}.
\fi
\begin{remark}
 Chevet's theorem states that \cite{halko2011finding}
\[\E\|SGT\| \leq \|S\| \|T\|_{F}+\|S\|_{F}\|T\|.\]
\cref{SquaredChevet} immediately implies Chevet's theorem by H{\"o}lder's Inequality. \end{remark}

\subsection{Proof of \cref{NysExpBound}}
\label{Section:NysExpBoundPf}
\begin{proof}
Proposition 11.1 in \cite[Sec. 11]{MartTroppSurvey} and the argument of Theorem 11.4  in \cite[Sec. 11]{MartTroppSurvey} shows that
\[\|A-\hat{A}_{\textrm{nys}}\| \leq \|\Sigma_{\ell-p+1}\|^{2}+\|\Sigma_{\ell-p+1}\Omega_{2}\Omega_1^{\dagger}\|^{2}.\]
Taking expectations and using $\|\Sigma_{\ell-p+1}\|^{2} = \lambda_{\ell-p+1}$ gives
\[\E\|A-\hat{A}_{\textrm{nys}}\| \leq \lambda_{\ell-p+1}+\E\|\Sigma_{\ell-p+1}\Omega_{2}\Omega_1^{\dagger}\|^{2}.\]
Using the law of total expectation, the second term may be bounded as follows
\begin{align*}
    \E\|\Sigma_{\ell-p+1}\Omega_{2}\Omega_1^{\dagger}\|^{2} &= \E\left(\E_{\Omega_{1}}\left[\|\Sigma_{\ell-p+1}\Omega_{2}\Omega_1^{\dagger}\|^{2}\right] \right) \\
    &\overset{(a)}{\leq}\E\left(\|\Sigma_{\ell-p+1}\|\hspace{2pt}\|\Omega_{1}^{\dagger}\|_{F}+\|\Sigma_{\ell-p+1}\|_{F}\hspace{2pt}\|\Omega_{1}^{\dagger}\| \right)^{2} \\
    &\overset{(b)}{\leq} 2\|\Sigma_{\ell-p+1}\|^{2}\E\|\Omega_{1}^{\dagger}\|^{2}_{F}+2\|\Sigma_{\ell-p+1}\|_{F}^{2}\E\|\Omega_{1}^{\dagger}\|^{2}\\
    &\overset{(c)}{\leq} \frac{2(\ell-p)}{p-1}\lambda_{\ell-p+1}+\frac{2\textrm{e}^{2}\ell}{p^{2}-1}\left(\sum_{j>\ell-p}\lambda_{j}\right),
\end{align*}
 where in step $(a)$ we use Squared Chevet (\cref{SquaredChevet}). In step $(b)$ we invoke the elementary identity $(a+b)^{2}\leq 2a^{2}+2b^{2}$, and in step $(c)$ we apply the bounds from \cref{GaussInvLemma}.
Inserting the above display into the bound for $\E\|A-\hat{A}_{\textrm{nys}}\|$ yields
\[\E\|A-\hat{A}_{\textrm{nys}}\| \leq \left(1+\frac{2(\ell-p)}{p-1}\right)\lambda_{\ell-p+1}+\frac{2\textrm{e}^{2}\ell}{p^{2}-1}\left(\sum_{j>\ell-p}\lambda_{j}\right).\]
As the bound above holds for any $2\leq p\leq \ell-2$,
we may take the minimum over admissible $p$ to conclude the result.
\end{proof}

\subsubsection{Proof of \cref{GenInvErrBnd}}
We require the following fact from  \cite[Chapter X]{bhatia2013matrix} ,
\begin{lemma}[\cite{bhatia2013matrix} Lemma X.1.4.]\label{fact-bhatia}
Let $X,Y$ be psd matrices. Then
\[\normiii{(X+I)^{-1}-(X+Y+I)^{-1}}\leq \normiii{Y(Y+I)^{-1}}\]
for every unitarily invariant norm.
\end{lemma}
\label{section:GenInvErrBndPf}
\begin{proof}[Proof of \cref{GenInvErrBnd}]
We first prove \eqref{eq-InvInequality}.
Under the hypotheses of \cref{GenInvErrBnd},
we may strengthen \cref{fact-bhatia} by scaling the identity to deduce
\[\|(\hat{A}+\mu I)^{-1}-(A+\mu I)^{-1}\| \leq \frac{1}{\mu}\|(E+\mu I)^{-1}E\|.\]

Recall that the function $f(t) = \frac{t}{t+\mu}$ is matrix monotone,
so that $A\preceq B$ implies $f(A)\preceq f(B)$. As $E\preceq \|E\|I$, it follows that
\[\|(\hat{A}+\mu I)^{-1}-(A+\mu I)^{-1}\| \leq \frac{1}{\mu}\frac{\|E\|}{\|E\|+\mu}.\]
Hence we have established the desired inequality.

Next we show the bound is attained when $\hat{A} = [A]_{\ell}$.
Applying the Woodbury identity, we may write
\[([A]_{\ell}+\mu I)^{-1} = V_{\ell}(\hat{\Lambda}_{\ell}+\mu I)^{-1}V_{\ell}^{T}+\frac{1}{\mu}(I-V_{\ell}V_{\ell}^{T}).\]
Using the eigendecomposition of $A = V_{\ell}\Lambda_{\ell}V_{\ell}^{T}+V_{n-\ell}\Lambda_{n-\ell}V_{n-\ell}^{T}$,
we obtain
\begin{align*}
([A]_{\ell}+\mu I)^{-1} - (A+\mu I)^{-1} &= \frac{1}{\mu}(I-V_{\ell}V_{\ell}^{T})-V_{n-\ell}(\Lambda_{n-\ell}+\mu I)^{-1}V_{n-\ell}^{T}\\
 &= \frac{1}{\mu}V_{n-\ell}(I-(\Lambda_{n-\ell}+\mu I)^{-1})V_{n-\ell}^{T} \\ &=\frac{1}{\mu}V_{n-\ell}\left(\Lambda_{n-\ell}(\Lambda_{n-\ell}+\mu I)^{-1}\right)V_{n-\ell}^{T}.
\end{align*}
Hence
\[\|(A+\mu I)^{-1}-([A]_{\ell}+\mu I)^{-1}\| = \frac{\lambda_{\ell+1}}{\mu(\lambda_{\ell+1}+\mu)}.\]
\end{proof}
\subsubsection{Proof of statements for the optimal low-rank preconditioner $P_{\star}$}
\label{section:PrecondPfs}
We show that $P_{\star}$ is the best symmetric positive definite preconditioner
that is constant off $V_{\ell}$.
\begin{lemma}
Let $\mathcal{P} = \{P: P = V_{\ell}MV_{\ell}^{T}+\beta(I-V_{\ell}V_{\ell}^{T})$\ \textrm{where}\ $\beta>0$\  \textrm{and}\ $M\in \mathbb{S}^{+}_{\ell}(\R)\}$.
With this parametrization, define $P_\star$ by setting $M= \frac{1}{\lambda_{\ell+1}+\mu}(\Lambda_\ell+\mu I)$ and $\beta = 1$.
Then for any symmetric psd matrix $A$ and $\mu \geq 0$,
\begin{eqnarray}
    \min_{P\in\mathcal{P}} \kappa_{2}(P^{-1/2}A_{\mu}P^{-1/2}) = \frac{\lambda_{\ell+1}+\mu}{\lambda_{n}+\mu}, \label{eqn-optimal-prec-value} \\
    P_\star = \argmin_{P\in\mathcal{P}} \kappa_{2}(P^{-1/2}A_{\mu}P^{-1/2}).
\end{eqnarray}
\end{lemma}
\begin{proof}
We first prove the lefthand side of \eqref{eqn-optimal-prec-value} is always at least as large as the righthand side,
and then show the bound is attained by $P_{\star}$.
Given $P \in \mathcal{P}$, we have
\[
P^{-1/2}A_{\mu}P^{-1/2} = V_{\ell}M^{-1/2}(\Lambda_{\ell}+\mu I)M^{-1/2}V_{\ell}^{T}+\frac{1}{\beta}V_{n-\ell}(\Lambda_{n-\ell}+\mu I)V_{n-\ell}^{T}.
\]
For any $1\leq i,j\leq n$,
\[
\kappa_{2}(P^{-1/2}A_{\mu}P^{-1/2}) = \frac{\lambda_{1}(P^{-1/2}A_{\mu}P^{-1/2})}{\lambda_{n}(P^{-1/2}A_{\mu}P^{-1/2})}\geq \frac{\lambda_i(P^{-1/2}A_{\mu}P^{-1/2})}{\lambda_j(P^{-1/2}A_{\mu}P^{-1/2})}.
\]
From our expression for $P^{-1/2}A_{\mu}P^{-1/2}$,
we see that $(\lambda_{\ell+1}+\mu)/\gamma,(\lambda_n+\mu)/\gamma$ are eigenvalues of $P^{-1/2}A_{\mu}P^{-1/2}$.
Hence for any $P\in \mathcal{P}$, the following inequality holds:
\[
\kappa_{2}(P^{-1/2}A_{\mu}P^{-1/2})\geq \frac{\lambda_{\ell+1}+\mu}{\lambda_n+\mu},
\]
proving \eqref{eqn-optimal-prec-value}.
Using the definition of $P_\star$, we see
\begin{align*}
P_{\star}^{-1/2}A_{\mu}P^{-1/2} &= (\lambda_{\ell+1}+\mu)V_{\ell}V_{\ell}^{T}+V_{n-\ell}(\Lambda_{n-\ell}+\mu I)V_{n-\ell}^{T}, \\
\kappa_2(P^{-1/2}_{\star}A_{\mu}P^{-1/2}_{\star}) &= (\lambda_{\ell+1}+\mu)/(\lambda_{n}+\mu).
\end{align*}
\end{proof}
\iftutorial
\subsection{Proof of Squared Chevet}
\label{section:SquaredChevetPf}
In this subsection we provide a proof of \cref{SquaredChevet}. The proof is based on a Gaussian comparison inequality argument, a standard technique in the high dimensional probability literature.
\begin{proof}
Let \[U = \{S^{T}a: \|a\|_{2} = 1\}\subset \R^{m}\]
\[V = \{Tb: \|b\|_{2} = 1\}\subset \R^{n}\]
and for $u\in U$, $v\in V$ consider the Gaussian processes
\begin{equation*}
   Y_{uv} = \langle u,Gv\rangle +\|S\|\hspace{2pt}\|v\|\gamma
\qquad \textup{and} \qquad
    X_{uv} = \|S\|\langle h,v \rangle +\|v\|\langle g,u\rangle,
\end{equation*}
where
\begin{itemize}
    \item $G\in \R^{m\times n}$ is a Gaussian random matrix,
    \item $g,h$ are Gaussian random vectors in $\R^m$ and $\R^n$ respectively,
    \item and $\gamma$ is $N(0,1)$ in $\R$.
\end{itemize}
Furthermore, $G,g,h$ and $\gamma$ are all independent.

A standard calculation shows that the conditions of Slepian's lemma (\cite[Corollary 3.12, p.~72]{ledoux2013probability}) are satisfied. Hence we conclude that
\begin{equation}
\label{eq-SlepIneq}
   \P\left(\max_{u,v}Y_{uv}>t\right)\leq \P\left(\max_{u,v}X_{uv}>t\right).
\end{equation}
We are now ready to prove \cref{SquaredChevet}.
Throughout the argument below, we use the notation $X_{+} = \max\{X,0\}$.

We first observe by Jensen's inequality with respect to $\gamma$ and the variational characterization of the singular values that
\begin{align*}
   \E\max_{u,v} (Y_{uv})^{2}_{+}
   &= \E\max_{\|a\| = 1, \|b\| = 1} \left( \langle S^{T}a,GTb\rangle+\|S\| \|Tb\|\gamma\right)^{2}_{+} \\
   &\geq \E_{G}\max_{\|a\| = 1, \|b\| = 1}\left( \langle S^{T}a,GTb\rangle^{2}\right)_{+} = \E_{G}\|SGT\|^{2}.
\end{align*}
Hence $\E_{G}\|SGT\|^{2}$ is majorized by $\E\max_{u,v}\hspace{2pt} (Y_{uv})^{2}_{+}.$
For $X_{uv}$, we note that
\begin{align*}
  \E\max_{u,v} (X_{uv})_{+}^{2}
  &\leq\E\max_{u,v}X_{uv}^{2} = \E\max_{\|a\|=1,\|b\| = 1}\left(\|S\|\langle h,Tb \rangle +\|Tb\|\langle g,S^{T}a\rangle\right)^{2} \\
  &\overset{(a)}{\leq }\E\left(\|S\|^{2}\|T^{T}h\|^{2}+2\|S\|\|T\|\|T^{T}h\|\|Sg\|+\|T\|^{2}\|Sg\|^{2}\right)\\
  &\overset{(b)}{\leq} \|S\|^{2}\|T\|^{2}_{F}+2\|S\|\|T\|\|S\|_{F}\|T\|_{F}+\|T\|^{2}\|S\|_{F}^{2} \\
  &= (\|S\|\|T\|_{F}+\|T\|\|S\|_{F})^{2},
  \end{align*}
where in step $(a)$ we expand the quadratic and use Cauchy-Schwarz. Step $(b)$ from a straightforward calculation and H{\"o}lder's inequality.

To conclude, we use integration by parts and \eqref{eq-SlepIneq} to obtain
\begin{align*}
\E_{G}\|SGT\|^{2}
&\leq \E\max_{u,v}\hspace{2pt} (Y_{uv})^{2}_{+} = \int^{\infty}_{0}t\P\left(\max_{u,v}(Y_{uv})_{+}>t\right)dt\\
&= \int^{\infty}_{0}t\P\left(\max_{u,v}Y_{uv}>t\right)dt \leq \int^{\infty}_{0}t\P\left(\max_{u,v}X_{uv}>t\right)dt \\
&= \int^{\infty}_{0}t\P\left(\max_{u,v}(X_{uv})_{+}>t\right)dt  = \E\max_{u,v} (X_{uv})^{2}_{+} \\ &\leq (\|S\|\|T\|_{F}+\|T\|\|S\|_{F})^{2}
\end{align*}
completing the proof.
\end{proof}
\else

\fi

\section{Additional experimental details}
\label{section:Experimentaldetails}
Here we provide additional details on the experimental procedure and the methods we compared to.
\subsection{Ridge regression experiments}
\label{section:RidgeRegressionDetails}
 Most of the datasets used in our ridge regression experiments are classification datasets.
 We converted them to regression problems by using a one-hot vector encoding.
 The target vector $b$ was constructed by setting $b_i=1$ if example $i$ has the first label and 0 otherwise.
 We did no data pre-processing except on CIFAR-10,
 where we scaled the matrix by $255$ so that all entries lie in $[0,1]$.
 
We now give an overview of the hyperparameters of each method.
The R\&T preconditioner has only one hyperparameter: the sketch size $\ell_{\textrm{RT}}$.
AdaIHS has five hyperparameters: $\rho,\lambda_{\rho},\Lambda_{\rho}, \mu_{\textrm{gd}}(\rho)$, and $c_{\textrm{gd}}(\rho)$. The hyperparameter $\rho\in (0,1)$ controls the remaining four hyperparameters, which are set to the values recommended in \cite{lacotte2020effective}.
For the regularization path experiments, $\ell_{\textrm{RT}}$ and $\rho$ were chosen by grid search to minimize the time taken to solve the linear systems over the regularization path. We chose $m_{\textrm{RT}}$ from the linear grid $jd$, where $j\in\{1,\dots,8\}$. Additionally, we restrict $j\leq 4$ for Guillermo as $jd\geq n$ when $j\geq 5$, and hence no benefit is gained over a direct method.  For AdaIHS, $\rho$ was chosen from the linear grid $\rho = j\times 10^{-1}$ where $j \in\{1,\dots,9\}$. We set the initial sketch size for AdaIHS to $100$ for both sets of experiments.

We reused computation as much as possible for both R\&T and AdaIHS, which we now detail. To construct the R\&T preconditioner, we incur a $O(nd\log(n)+\ell_{\textrm{RT}}d^{2})$ to cache the Gram matrix and pay an $O(d^{3})$ to update the preconditioner for each value of $\mu$. In the case of AdaIHS, for each value of $\mu$ we cache the sketch $SA$ and the corresponding Gram matrix. We then use them for the next value of $\mu$ on the path until the adapativity criterion of the algorithm deems a new sketch necessary. For AdaIHS computing the sketch only costs $O(nd\log(n)).$

We now give the details of the random features experiments. For Shuttle-rf we used random features corresponding to a Gaussian kernel with bandwidth parameter $\sigma = 0.75$, we set $\mu = 10^{-8}/n$. For smallNORB-rf we used ReLU random features with $\mu = 6\times 10^{-4}$. For Higgs we normalized the features by their z-score and we used random features for a Gaussian Kernel with $\sigma = 5$ and regularization $\mu = 10^{-4}$.  Similarly for YearMSD, we normalized the matrix by their z-score and used random features for a Gaussian kernel with $\sigma = 8$ and $\mu = 10^{-5}$. The sketch size for R\&T was selected from $\{d, 2d\}$, to prevent the cost of the forming and applying preconditioner from becoming prohibitive. We selected the AdaIHS parameter $\rho$ from the same grid used for the regularization path experiments. We also capped the sketch size for AdaIHS for each dataset by the sketch sized used for R\&T.     

Finally, we give the details of our implementation of Nystr{\"o}m PCG. For both sets of experiments we used \cref{alg:AdaRandNysAppx} initialized at $\ell = 100$, with an error tolerance of $30\mu$, and $q = 5$ power iterations. To avoid trivialities, the rank of the preconditioner is capped at $\ell_{\textrm{max}} = 0.5d$ for CIFAR-10 and $\ell_{\textrm{max}} = 0.4d$ for Guillermo. For the random features experiments we capped $\ell$ at $\ell_{\textrm{max}} = 2000$. In the regularization path experiments, we keep track of the latest estimate $\hat{E}$ of $\|E\|$, and do not compute a new Nystr{\"o}m approximation unless $\hat{E}$ is larger than the error tolerance for the new regularization parameter. When we compute the new Nystr{\"o}m approximation, the adaptive algorithm is initialized with a target rank of twice the old one.

The values of hyperparameters used for all experiments are summarized in \cref{t-parameters}.
\begin{table}[t]
\footnotesize
\begin{center}
    \begin{tabular}{|c|c|c|c|c|}
    \hline
    \thead{Dataset} & \thead{(R\&T) sketch \\ size} & \thead{AdaIHS rate} & \thead{Initial AdaIHS \\sketch size} & \thead{Initial \\ Nystr{\"o}m \\ sketch size} \\
    \hline
    CIFAR-10 & $3d$ & $\rho$ = 0.3 &  $100$ & $100$\\
    \hline
    Guillermo & $d$ & $\rho$ = 0.3 & $100$ & $100$ \\
    \hline
    shuttle-rf  & $2d$ & $\rho$ = 0.1 & $100$ & $100$ \\
    \hline
    smallNORB-rf & $2d$ & $\rho$ = 0.3 & $100$ &  $100$  \\
    \hline
    \end{tabular}
    \caption{\label{t-parameters} Ridge regression experimental parameters.}
\end{center}
\end{table}

\subsection{ALOOCV}
The datasets were chosen so that $n$ and $d$ are both large, the challenging regime for ALOOCV. The first three datasets are binary classification problems, while SVHN has multiple classes. For SVHN we created a binary classification problem by looking at the first class vs. remaining classes.

For the large scale problems the adaptive algorithm for Nystr{\"o}m PCG was initialized at $\ell_{0} = 500$ and is capped at $\ell_{\textrm{max}} = 4000$.  We set the solve tolerances for both algorithms to $10^{-10}$. As before, we sample 100 points randomly from each dataset.

\subsection{Kernel Ridge Regression}

We converted the binary classification problem
to a regression problem by constructing the target vector as follows:
We assign +1 to the first class and -1 to the second class.
For multi-class problems, we do one-vs-all classification; this formulation leads to multiple right hand sides, so we use block PCG for both methods. 
We did no data pre-processing except for EMNIST, MiniBooNE, MNIST, and Santander. 
For EMNIST and MNIST the data matrix was scaled by $255$ so that its entries lie in $[0,1]$, while for MinBooNE and Santander the features were normalized by their z-score.
The number of random features, $m_{\textrm{rf}}$ from the linear grid $m_{\textrm{rf}} = j\times 10^{3}$ for $j = 1,\dots,9$.
For adaptive Nystr{\"o}m PCG  we capped the maximum rank for the preconditioner at $\ell_{\max} = \lfloor0.1n\rfloor$ and used a tolerance of $40$ for the ratio $\hat{\lambda}_{\ell}/{n\mu}$ on all datasets.

\section{Additional numerical results}
\label{section:AddNumerics}
Here we include some additional numerical results not appearing in the main paper.
\subsection{ALOOCV}
\cref{t-rcv1-realsim-results-part-ii} contains more details about the preconditioner and preconditioned system for the large scale ALOOCV experiments in \cref{section:ALOOCV}. The original condition number in \cref{t-rcv1-realsim-results-part-ii} below is estimated as follows. First we compute the top eigenvalue of the Hessian using Matlab's eigs() command, then we divide this by $\mu$.
\begin{table}[H]
\footnotesize
\begin{center}
    \begin{tabular}{|c|c|c|c|c|}
    \hline
    \textbf{Dataset} & \thead{Nystr{\"o}m \\ rank} & \thead{Preconditioner\\ construction \\ time(s)} & \thead{Condition \\ number\\ estimate} & \thead{Preconditioned\\ condition \\ number\\ estimate} \\
    \hline
    rcv1 $(\mu = \num{1e-4})$ & $1000$  & $19.5\hspace{2pt}(0.523)$ & $21.6$ & $2.98\hspace{2pt} (0.081)$\\
    \hline
    rcv1 $(\mu = \num{1e-8})$ & $4000$ & $100.6\hspace{2pt}(3.46)$ & $5.70\mathrm{e}{+3}$ & $17.2\hspace{2pt} (0.218)$ \\
    \hline
    realsim $(\mu = \num{1e-4})$ & $3100\hspace{2pt}(1.41\mathrm{e}{+3})$ & $82.01 (2.04)$ & $10.0$ & $1.70 \hspace{2pt}(0.2324)$\\
    \hline
    realsim $(\mu = \num{1e-8})$ & $4000$ & $108.3\hspace{2pt}(6.21)$ & $2.13\mathrm{e}{+4}$ & $62.4\hspace{2pt}(0.945)$ \\
    \hline
    \end{tabular}
    \caption{\label{t-rcv1-realsim-results-part-ii} For $\mu = 10^{-4}$ the Hessian is well-conditioned for both datasets, so there is little value to preconditioning. For $\mu = 10^{-8}$, the ill-conditioning of the Hessian increases significantly, making preconditioning more valuable. Furthermore, as ALOOCV uses Block PCG on at least several batches of data points, the cost of constructing the preconditioner is negligible compared to the cost of solving the linear systems (see \cref{t-rcv1-realsim-results-part-i} in \cref{section:ALOOCV}). }
\end{center}
\end{table}

\section{Adapative rank selection via a-posteriori error estimation}
Below we give the pseudocode for the algorithms used to perform adaptive rank selection using the strategy proposed in \cref{section:ErrorRankDoubling}.
\label{section:AdaRankSelection}
\subsection{Randomized Powering algorithm}
The pseudo-code for estimating $\|E\|$ by the randomized power method is given in \cref{alg:RandPowErrEst}
\begin{algorithm}[t]
	\caption{Randomized Power method for estimating $\|E\|$}
	\label{alg:RandPowErrEst}
	\begin{algorithmic}[1] 
		\Require{symmetric PSD matrix $A\in \R^{n\times n}$, approximate eigenvectors $U$, approximate eigenvalues $\hat{\Lambda}$, and number of power iterations $q$.}
		\State $g = \textrm{randn}(n,1)$ 
		\State $v_0 = \frac{g}{\|g\|_{2}}$
		    \For{$i = 1,\dots,q$}
		        \State $v = Av_0-U(\hat{\Lambda}(U^{T}v_{0}))$
		        \State $\hat{E} = v_{0}^{T}v$
		        \State $v = \frac{v}{\|v\|_{2}}$
		        \State $v_{0}\leftarrow v$
		    \EndFor
		\Ensure{estimate $\hat{E}$ of $\|E\|$}
	\end{algorithmic}
\end{algorithm}

\subsection{Adaptive rank selection algorithm}
The pseudocode for adaptive rank selection by a-priori error estimation is given in \cref{alg:AdaRandNysAppx}. The code is structured to reuse use the previously computed $\Omega$ and $Y$, resulting in significant computational savings. The error $\|E\|$ is estimated from $q$ iterations of the randomized power method on the error matrix $A-U\hat{\Lambda}U^{T}$.

\begin{algorithm}[t]
	\caption{Adaptive Randomized Nystr{\"o}m Approximation}
	\label{alg:AdaRandNysAppx}
	\begin{algorithmic}[1] 
		\Require{symmetric psd matrix $A$, initial rank $\ell_0$, maximum rank $\ell_{\textrm{max}}$, number of power iterations for estimating $E$, error tolerance $\textrm{Tol}$ }
		\State $Y = [\hspace{5pt}], \Omega = [\hspace{5pt}],$ and $E = \textrm{Inf}$
		\State $m = \ell_{0}$
		    \While {$E>\textrm{Tol}$}
		        \State Generate Gaussian test matrix $\Omega_{0}\in \mathbb{R}^{n\times m}$
		        \State $[\Omega_{0},\sim] = \textrm{qr}(\Omega_{0},0)$
		        \State $Y_{0} = A\Omega_{0}$
		        \State $\Omega = [\Omega\hspace{5pt}\Omega_{0}]$ and $Y = [Y\hspace{5pt}Y_{0}]$
		        \State $\nu = \sqrt{n}\hspace{2pt}\textrm{eps}(\textrm{norm}(Y,2))$
		        \State $Y_{\nu} = Y+\nu\Omega,\hspace{3pt}$
		        \State $C = \textrm{chol}(\Omega^{T}Y_{\nu})$
		        \State $B = Y_{\nu}/C$
		        \State Compute $[U,\Sigma,\sim] = \textrm{svd}(B,0)$
		        \State $\hat{\Lambda} = \max\{0,\Sigma^{2}-\nu I\}$ \Comment{remove shift}
		        \State $E = \textrm{RandomizedPowerErrEst}(A,U,\hat{\Lambda},q)$
                \Comment{estimate error}
		        \State $m\leftarrow \ell_{0}$, $\ell_{0} \leftarrow 2\ell_{0}$
                \Comment{double rank if tolerance is not met}
		            \If{$\ell_{0}>\ell_{\textrm{max}}$}
		                \State $\ell_{0} = \ell_{0}-m$  \Comment{when $\ell_{0}>\ell_{\textrm{max}}$, reset to $\ell_{0} = \ell_{\textrm{max}}$}
		                \State $m = \ell_{\textrm{max}}-\ell_{0}$
		                \State Generate Gaussian test matrix $\Omega_{0}\in \mathbb{R}^{n\times m}$
		                \State $[\Omega_{0},\sim] = \textrm{qr}(\Omega_{0},0)$
		                \State $Y_{0} = A\Omega_{0}$
		                \State $\Omega = [\Omega\hspace{5pt}\Omega_{0}]$ and $Y = [Y\hspace{5pt}Y_{0}]$
		                \State $\nu = \sqrt{n}\hspace{2pt}\textrm{eps}(\textrm{norm}(Y,2))$ \Comment{compute final approximation and break}
		                \State $Y_{\nu} = Y+\nu\Omega,\hspace{3pt}$
		                \State $C = \textrm{chol}(\Omega^{T}Y_{\nu})$
		                \State $B = Y_{\nu}/C$
		                \State Compute $[U,\Sigma,\sim] = \textrm{svd}(B,0)$
		                \State $\hat{\Lambda} = \max\{0,\Sigma^{2}-\nu I\}$
		                \State \textbf{break}
		            \EndIf
		    \EndWhile
		\Ensure{Nystr{\"o}m approximation $(U,\hat{\Lambda})$}
	\end{algorithmic}
\end{algorithm}

\end{document}